\documentclass[11pt,a4paper]{article}

\usepackage[english]{babel}
\usepackage[utf8]{inputenc}
\usepackage[T1]{fontenc}

\usepackage[top=3cm,bottom=2.5cm,left=3cm,right=3cm,marginparwidth=1.75cm]{geometry}
\usepackage{bigstrut}
\usepackage{tikz}
\usepackage{amsmath}
\usepackage{hhline}
\usepackage{xcolor}
\usepackage{shuffle}
\usepackage{wasysym}
\usepackage{amsthm}
\usepackage{amssymb}
\usepackage{mathtools}
\usepackage{MnSymbol}
\usepackage{cancel}
\usepackage{graphicx}
\usepackage[colorinlistoftodos]{todonotes}
\usepackage[colorlinks=true, allcolors=blue]{hyperref}
\usepackage[all]{xy}
\usepackage{comment}
\newcommand{\p}{\mathbb{P}}

\newcommand{\N}{\mathbb{N}}
\newcommand{\Z}{\mathbb{Z}}
\newcommand{\C}{\mathbb{C}}
\newcommand{\R}{\mathbb{R}}
\newcommand{\RR}{\mathcal{R}}
\newcommand{\U}{\mathcal{U}}
\newcommand{\V}{\mathcal{V}}
\newcommand{\LL}{\mathcal{L}}
\newcommand{\AAA}{\mathcal{A}}
\newcommand{\f}{\varphi}
\newcommand{\G}{\mathcal{G}}
\newcommand{\Lie}{\mathop{\rm Lie}\nolimits}
\newcommand{\conv}{\mathop{\rm Conv}\nolimits}
\newcommand{\lcm}{\mathop{\rm lcm}\nolimits}

\newcommand{\inv}{\mathop{\rm inv}\nolimits}
\DeclareMathOperator{\vol}{vol}
\DeclareMathOperator{\sgn}{sgn}
\DeclareMathOperator{\Pf}{Pf}
\newtheorem{theorem}{Theorem}[section]
\newtheorem{proposition}[theorem]{Proposition}
\newtheorem{lemma}[theorem]{Lemma}
\newtheorem{corollary}[theorem]{Corollary}
\newtheorem{remark}[theorem]{Remark}

\newtheorem{example}[theorem]{Example}
\newtheorem{notation}[theorem]{Notation}
\newtheorem{conjecture}[theorem]{Conjecture}
\newtheorem*{question}{Question}

\newtheorem{definition}[theorem]{Definition}
\DeclareMathOperator{\Spec}{Spec}
\def\CC{\mathbb{C}}

\begin{document}
\title{Toric geometry of path signature varieties}
\author{Laura Colmenarejo\footnote{Laura Colmenarejo, University of Massachusetts at Amherst (Amherst, US), \newline \texttt{laura.colmenarejo.hernando@gmail.com}}, 
Francesco Galuppi\footnote{Francesco Galuppi, University of Trieste (Trieste, Italy), \texttt{fgaluppi@units.it}},
Mateusz Micha\l{}ek\footnote{Mateusz Micha\l{}ek, Max Planck Institute for Mathematics in the Sciences (Leipzig, Germany), \texttt{michalek@mis.mpg.de}, and Polish Academy of Sciences,  
Institute of Mathematics (Warsaw, Poland), and Aalto University (Aalto, Finland). MM was supported by the Polish National Science Center Project 2013/08/A/ST1/00804 affiliated at the University of Warsaw.}
\thanks{The authors would like to thank Lara Bossinger for the discussion in the early stage of this project, Avinash Kulkarni for the helpful suggestions and Joscha Diehl for his feedback.}}
\maketitle

\begin{abstract}
In stochastic analysis, a standard method to study a path is to work with its signature. This is a sequence of tensors of different order that encode information of the path in a compact form. When the path varies, such signatures parametrize an algebraic variety in the tensor space. The study of these signature varieties builds a bridge between algebraic geometry and stochastics, and allows a fruitful exchange of techniques, ideas, conjectures and solutions.

In this paper we study the signature varieties of two very different classes of paths. The class of rough paths is a natural extension of the class of piecewise smooth paths. It plays a central role in stochastics, and its signature variety is toric. The class of axis-parallel paths has a peculiar combinatoric flavour, and we prove that it is toric in many cases.
\end{abstract}

\section{Introduction}\label{sec:intro}

A \emph{path} is a continuous map $X\colon [0,1]\to\R^d$. This very simple mathematical object can be used to interpret a wide range of situations. From a physical transformation to a meteorological model, from a medical experiment to the stock market, everything that involves parameters changing with time can be described by a path. This makes paths priceless tools in many branches of mathematics, as well as in a number of applied sciences. The downside is that, being a continuous object, explicit computations on a path are not easy to handle. A typical way to overcome this problem %, when we study something complicated,
 is to find invariants that are simpler to understand and %that
  can provide %hopefully give
   us enough information. %Any mathematician has met several examples: the genus of a Riemann surface, the moments of a random variable, the fundamental group of a topological space, and so on.

For paths, this problem was faced in \cite{Chen}. Assume that $X$ is piecewise smooth, and fix $k\in\N$. Let $X_i$ be the composition of $X$ with the projection to the $i$-th coordinate. Chen defined the \emph{$k$-th signature} of $X$ to be the order $k$ tensor $\sigma^{(k)}(X)$ whose $(i_1\ldots i_k)$-th entry is %the real number given by the iterated integral
\[\int_{0}^{1}\int_{0}^{t_k}\dots\int_{0}^{t_3}\int_{0}^{t_2}\dot{X}_{i_1}(t_1)\cdot\ldots\cdot\dot{X}_{i_k}(t_k)dt_1\dots dt_k.
\]
By convention, $\sigma^{(0)}(X)=1$. The sequence
\[\sigma(X):=(\sigma^{(k)}(X)\mid k\in\N)\]
is called the \emph{signature} of $X$. Sometimes it is also convenient to consider a truncated signature $\sigma^{\le m}(X):=(\sigma^{(k)}(X)\mid k\in\{0,\dots,m\})$.
In this context, it is natural to ask how much these tensors tell us about $X$. In \cite[Theorem 4.1]{chenuniqueness} Chen proved that, up to a mild equivalence relation, the signature allows to uniquely recover a piecewise smooth path. %This means that signatures are useful tools to encode in a compact form the information carried by paths.

Signatures are  appreciated not only in stochastics, but also in topological data analysis, financial mathematics or machine learning among others (see \cite{GyurkoLyonsKontkowskiField,ChevyrevKormilitzin, ChevyrevNandaOberhauser, ChevyrevOberhauser}). Chen's iterated integrals have deep connections with de Rham homotopy theory, as shown in \cite{Hain2002}. In this paper we are interested in the geometric side, illustrated in Section \ref{subsec:geometric side}. We study the so-called \emph{signature varieties} -- roughly speaking the geometric locus of all possible signatures. In Section \ref{sec:rough veronese} we consider the variety $\RR_{d,k,m}$ of signatures of rough paths.
Our results often rely on toric geometry, which allows us to answer several questions about their equations and singularities:
\begin{itemize}
\item Answering a question raised in \cite{AFS18} and \cite{G18}, we show that the ideal of $\RR_{d,k,m}$ may be generated in arbitrary high degree - Proposition \ref{pro:rough veronese is not gen in bounded degree}. However, we prove that quadratic polynomials define a (possibly reducible) variety, of which $\RR_{d,k,m}$ is an irreducible component - Proposition \ref{prop:quad}.
\item We characterize the cases in which $\RR_{d,k,m}$ is an embedding of the weighted projective space. We find conditions that make it (projectively) normal and examples when it is not - Subsection \ref{sub:GeomRVV}.
\item We study the  asymptotic behavior of the degree of $\RR_{d,k,m}$ and we give explicit formulas for some special cases - Subsection \ref{sub:degRVV}.
\end{itemize} 

In Section \ref{sec:axis parallel} we study the geometry of the signature variety $\AAA_{\nu,k}$ of axis parallel paths:
\begin{itemize}
\item We provide an easy, combinatorial parametrization of the variety - Lemma \ref{lem:CombDescription}.
\item We apply the above description to prove that $\AAA_{\nu,k}$ is toric in several cases - Section \ref{sub: axis toricness}. We study the dimension of $\AAA_{\nu,k}$ exhibiting defective cases - Section \ref{sub:axisDim}.
\end{itemize}

Finally, we use our knowledge of axis paths to prove a general formula for the determinant of the signature matrix of any path, Theorem \ref{thm:det is a square} and Corollary \ref{cor:shuffle determinant is a shuffle square}.

In order to present those results we need to recall the algebraic background.
%describe those varieties, we recall the algebraic background.

\subsection{The tensor algebra}
The $k$-th signature of a path $X$ belongs to $(\R^d)^{\otimes k}$. We now introduce an ambient space for the whole signature $\sigma(X)$.

\begin{definition}
	The \emph{tensor algebra} over $\R^d$ is the graded $\R$-vector space
	\[T((\R^d)):= %\bigoplus_{k\in\N}(\R^d)^{\otimes k}	=
	\R\times\R^d\times(\R^d)^{\otimes 2}\times\ldots\]
	of formal power series in the non-commuting variables $x_1,\dots,x_d$. It is an $\R$-algebra with respect to the tensor product%, and we denote by $p_k:T((\R^d))\to(\R^d)^{\otimes k}$ the projection
	. The algebraic dual of $T((\R^d))$ is the graded free $\R$-algebra
	\[T(\R^d)=\R\langle x_1,\dots,x_d\rangle\]
	of polynomials in the non-commuting variables $x_1,\dots,x_d$. %It is the unique free algebra over $x_1,\dots,x_d$.
\end{definition}

These spaces and their rich algebraic structures are well studied. In this section we recall the features we need, and we refer to \cite{Reutenauer} for a detailed treatment. Next result, proven in \cite[Section 2]{C57}, gives a first taste of the correlation between the tensor algebra and signatures of paths. Recall that the \emph{concatenation} of two paths $X$ and $Y$ is the path $X\sqcup Y\colon [0,1] \longrightarrow \mathbb{R}^d$ given by
\[(X\sqcup Y)(t)=\begin{cases} X(2t) & \mbox{ if } t\in \left[ 0,\frac{1}{2}\right], \\
X(1)+Y(2t-1) & \mbox{ if } t\in \left[\frac{1}{2},1\right].
\end{cases}
\]
%$X$ on $[0,1]$ and by $Y_{\cdot - 1} - Y_0 + X_1$ on $[1,2]$ (i.e. take $Y$, move it back to $0$ and then move it to the end of $X$). 
\begin{lemma}[Chen's identity]\label{lem:Chen's identity}
	If $X, Y\colon [0,1]\to\R^d$ are piecewise smooth paths, then $\sigma(X\sqcup Y)=\sigma(X)\otimes\sigma(Y)$ as formal power series in $T((\R^d))$.
\end{lemma}

We now introduce a useful notation.

\begin{notation}
	For $T\in T((\R^d))$, denote by $T_{i_1\dots i_k}$ the $(i_1\dots i_k)$-th entry of the order $k$ element of $T$. For $y\in\R$, we set $T_y((\R^d)):=\{T\in T((\R^d))\mid T_0=y\}$ to be the space of tensor sequences with constant entry 0. Moreover, we will identify a degree $k$ monomial $x_{i_1}\cdot\ldots\cdot x_{i_k}$ with the word $w=i_1\ldots i_k$ in the alphabet $\{1,\dots,d\}$. The number $k$ is called the \emph{length} of $w$ and it is denoted by $\ell(w)$. The degree 0 monomial corresponds to the empty word $e$. %We will write letters in bold in order to distinguish the number 1 from the letter ${\bf 1}=x_1$. 
	In this way, the tensor product of the two monomials corresponding to the words $v$ and $w$ is the word obtained by writing $v$ followed by $w$, and it is called the \emph{concatenation product}. The natural duality pairing
	\[\langle-,-\rangle\colon T((\R^d))\times T(\R^d)\to\R
	\]
	is given by $\langle T, i_1\ldots i_k\rangle=T_{i_1\dots i_k}$, and extended by linearity.
\end{notation}

Besides the concatenation of words, there is another product on $T(\R^d)$. It has a strong combinatoric flavour, but it will also allow us to define very interesting algebraic objects.

\begin{definition}
	%The \emph{shuffle product} of two words $v$ and $w$, denoted by $v\shuffle w$, is the sum  of all order-preserving interleavings of them. A more precise, recursive definition can be found in \cite[Section 1.4]{Reutenauer}. Let $w$, $w_1$ and $w_2$ be three words and $a$ and $b$ two letters. Also denote by $e$ the empty word. 
	We define the \emph{shuffle product} of two words recursively by
\begin{eqnarray*}
\begin{array}{l}
e \shuffle w = w \shuffle e = w, \text{ and }\\
(w_1\otimes a) \shuffle (w_2\otimes b) = \left(w_1\shuffle (w_2\otimes b)\right)\otimes a + \left((w_1\otimes a )\shuffle w_2\right)\otimes b.
\end{array}
\end{eqnarray*}
Less formally, $v\shuffle w$ is the sum  of all order-preserving interleavings of $v$ and $w$.
\end{definition}

For instance, $1\shuffle 23=123+213+231$. Despite its apparently complicated definition, the shuffle product enjoys good properties. For instance, $(T(\R^d),\shuffle, e)$ is a commutative algebra. Moreover next Lemma, proven in \cite[Proof of Corollary 3.5]{Reutenauer}, shows that the shuffle product behaves nicely with respect to the signatures.

\begin{lemma}[Shuffle identity]\label{lem:shuffle identity}
	If $X%:[0,1]\to\R^d
	$ is a piecewise smooth path, then
	\[\langle \sigma(X),v\rangle\cdot\langle \sigma(X),w\rangle=\langle \sigma(X),v\shuffle w\rangle
	\]
	for all words $v,w\in T(\R^d)$.
\end{lemma}

It follows that signatures do not fill the tensor space $T((\R^d))$, but rather they live in a subset. In order to make this observation precise, we introduce an important subspace of the tensor algebra.

\begin{definition}
	Consider the Lie bracketing $[T,S]=TS-ST$ on $T((\R^d))$. We define $\Lie(\R^d)\subset T_0((\R^d))$ to be the free Lie algebra generated by $x_1,\dots,x_d$, that is, the smallest vector subspace of $T((\R^d))$ that contains $x_1,\dots,x_d$ and is closed with respect to the bracketing.
\end{definition}
%The Lie algebra $\Lie(\R^d)$ is a linear subspace of $T((\R^d))$. 
%What we are really interested in 
The Lie group associated to $\Lie(\R^d)$ is an important object. One way to define it is as the image of $\Lie(\R^d)$ under the exponential map.
\begin{definition} Define $\exp:T_0((\R^d))\to T_1((\R^d))$ by the formal power series
	\[\exp(T):=\sum_{n=0}^{\infty}\frac{T^{\otimes n}}{n!}.\]
	We denote $\G(\R^d):=\exp(\Lie(\R^d))$.
\end{definition}

By construction, $(\G(\R^d),\otimes, e)$ is a group. However, there is another way to characterize it in terms of shuffle product. By \cite[Theorem 3.2]{Reutenauer},
	\[\G(\R^d)=\{T\in T_1((\R^d))\mid \langle T,v\rangle\cdot\langle T,w\rangle=\langle T,v\shuffle w\rangle
	\mbox{ for all words } v,w\in T(\R^d)\}.\]
Now we see the first clear connection with the signatures. By the shuffle identity, $\G(\R^d)$ contains the signatures of all piecewise smooth paths.

For our purposes, we need to point out that every definition we recalled has a truncated version. Namely, one can fix $m\in\N$ and consider
\[T^m(\R^d):=\bigoplus_{k=0}^m(\R^d)^{\otimes k},\]
where tensors of order greater than $m$ are set to zero. Inside $T^m(\R^d)$ there are $\G^m(\R^d)$ and $\Lie^m(\R^d)$.
The restriction of the map $\exp$ is defined in the same way, and we will write $\exp^{(m)}$ to denote the map $T_0^m((\R^d))\to T_1^m((\R^d))$.

Another feature of the group $\G(\R^d)$ is that, for every $m\in\N$, $\G^m(\R^d)$ not only contains, but actually coincides with the set of all truncated signatures of smooth paths (see \cite[Theorem 4.4]{AFS18}) and it is a Lie group (see \cite[Theorem 7.30]{FrizVictoir2010}).
Moreover, it is defined by finitely many polynomials - namely, the shuffle relations with words of length at most $m$ - hence it is an algebraic variety. It is irreducible by \cite[Theorem 4.10]{AFS18}. On the other hand, $\Lie^m(\R^d)$ is a finite dimensional vector space. In order to provide it with a basis, it is time to introduce another powerful combinatoric concept.

\begin{definition}
	A non-empty word $w$ is \emph{Lyndon} if, whenever we write $w=pq$ as the concatenation of two nonempty words, we have $w<q$ in the lexicographic order. %We denote by $W_{d,m}$ the set of Lyndon words of length at most $m$ in the alphabet $\{{\bf 1},\dots,{\bf d}\}$. 
	In this case there is a unique pair $(p,q)$ of nonempty words such that $w=pq$ and $q$ is minimal with respect to lexicographic order. The \emph{bracketing} of $w$ is $[p,q]=pq-qp$.
\end{definition}

As shown in \cite[Corolllary 4.14]{Reutenauer} and \cite[Proposition 4.7]{AFS18}, the bracketings of all Lyndon words of length at most $m$ form a basis for  $\Lie^m(\R^d)$. Therefore, in order to compute its dimension - and thus the dimension of the variety $\G^m(\R^d)$ - it is enough to count Lyndon words. First, recall that the M\"{o}bius function $\mu:\N\to\N$ sends a natural number $t$ to
\[\mu(t)=\begin{cases}
0 & \mbox{ if $t$ is divisible by the square of a prime},\\
1 & \mbox{ if $t$ is the product of an even number of distinct primes},\\
-1 & \mbox{ if $t$ is the product of an odd number of distinct primes}.\\
\end{cases}
\]
Then the number of Lyndon words of length $l$ in the alphabet $\{ 1,\dots,d\}$ %, denoted by $\mu_{l,d}$,
 is 
%\[\mu_{l,d}=
$\sum_{t\mid l}\frac{\mu(t)}{l}d^{\frac{l}{t}},$
and therefore %, as a vector space, $\Lie^m(\R^d)$ has dimension
\[\dim\Lie^m(\R^d)=\sum_{l=1}^m \sum_{t\mid l}\frac{\mu(t)}{l}d^{\frac{l}{t}}.\]

\subsection{Signatures from a geometric viewpoint}\label{subsec:geometric side}
Chen's theorem assumes the knowledge of the $k$-th signature for every $k\in\N$. What happens when we know only one of them?
In \cite{AFS18}, the authors consider the problem from an algebraic geometry perspective. If we fix a certain class of paths and the order $k$ of the tensors, then the $k$-th signature $\sigma^{(k)}$ is an algebraic map into $(\R^d)^{\otimes k}$. The closure of the image of  $\sigma^{(k)}$ is called the \emph{$k$-signature variety}. %is the closure of the set of all $k$-th signatures of paths of the chosen class. 
We get a first example by considering the class of all smooth paths in $\R^d$. If $X$ is smooth, then the $k$-th signature of $X$ is just the $k$-th entry of $\sigma(X)\in\G(\R^d)$. This leads to the following definition.
\begin{definition}
	The \emph{universal variety} $\U_{d,k}\subset(\R^d)^{\otimes k}$ is the projection of $\G(\R^d)$ onto the $k$-th factor. It is the closure of the set of all $k$-th signatures of smooth paths.
\end{definition}
Since every $\G^m(\R^d)$ is an irreducible variety, $\U_{d,k}$ is irreducible as well. By \cite[Theorem 6.1]{AFS18}, its dimension coincides with $\dim\G^k(\R^d)$. Equivalently, we can compute it as the dimension of the vector space $\Lie^k(\R^d)$, that is the number of Lyndon words of length at most $k$ in the alphabet $\{1,\dots,d\}$.
%Further examples of signature varieties are $\mathcal{P}_{d,k,m}$, which parametrizes $k$-th signatures of polynomial paths of degree at most $m$, and $\LL_{d,k,m}$, which parametrizes $k$-th signatures of piecewise linear paths with $m$ steps. They are both subvarieties of $\U_{d,k}$. It is worth to point out that, since every smooth path can be approximated by a sufficiently high degree polynomial path, or by a piecewise linear path with many steps, $\U_{d,k}=\mathcal{P}_{d,k,m}=\LL_{d,k,m}$ for large enough $m$.

This geometric approach provides a way to translate questions about paths into questions about the signature variety. For instance, the crucial problem of reconstructing uniquely the path from its $k$-signature translates into the injectivity of $\sigma^{(k)}$. This leads to study the fibers of $\sigma^{(k)}$ and therefore the dimension of the signature variety. Provided that a preimage can be reconstructed, such computation consists in solving a system of polynomial equations, and we can get an idea of how difficult that is by looking at the degree of the variety or at the degrees of the generators of its ideal. Taking a step further, the singular locus can point out that some signatures are different from the others, so we have a meaningful way to distinguish special paths in our class.

There are many interesting questions when it comes to study a signature variety, and we do not hope to give full answers to all of them. %However, some classes are better understood than others. For instance, $\mathcal{P}_{d,k,m}$ and $\LL_{d,k,m}$ are defined and studied in \cite{AFS18}. Moreover, in \cite{G18}, the author drops the smoothness hypothesis and investigates the class of rough paths. 
Path recovery from the third signature is the main goal of \cite{PSS18}. Further interplay between the tensor algebra and the signatures is explored in \cite{ColmenarejoPreiss}. Our contribution to this topic is a detailed study of the rough paths signature variety, presented in Section \ref{sec:rough veronese}, as well as the analysis of the axis parallel signature variety, which we describe in Section \ref{sec:axis parallel}. Both varieties present several geometric subtleties. When dealing with them, we could take advantage of a surprisingly rich combinatorial structure, that allows us to integrate a projective geometric approach with more computational techniques.

\section{The rough Veronese variety}\label{sec:rough veronese}

Researchers in stochastic analysis usually work with paths that are not piecewise smooth, and therefore do not have a signature in the sense of Chen's definition. One of the most important examples is the class of rough paths. The interest in rough paths is quickly growing. Applications include the study of controlled ODEs and stochastic PDEs (see \cite{LyonsCaruanaLevy2004}) as well as sound compression (see \cite{LyonsSidorova2005}), not to mention mathematical physics (see \cite{FrizGassiatLyons2015}). The main references on rough paths are \cite{FrizVictoir2010} and \cite{FrizHairer2014}. In this section we recall what the signature of a rough path is, and study the corresponding signature variety. %with an algebraic geometric approach.

The signature variety of all rough paths is the whole universal variety. In the same way as polynomial paths are an interesting subclass of smooth paths, in \cite[Section 5.4]{AFS18} the authors consider a nice subclass of rough paths, parametrized by $\Lie^m(\R^d)$. Their signature variety $\RR_{d,k,m}$ exhibits similarities with the classical Veronese variety and it was therefore named the \emph{rough Veronese variety} - see Definition \ref{def:rough veronese}. Another reason to study $\RR_{d,k,m}$ is that we can deduce many properties of the universal varieties from it. Indeed, $\U_{d,k}=\RR_{d,k,m}$ for every $m\ge k$ and moreover, $\U_{d,k}=\RR_{d,k,k}$ is a cone over $\RR_{d,k,k-1}$ (see \cite[Proposition 26]{G18}).
%Our main results in this section are the following.

\subsection{Preliminaries}

One purpose of the definition of a rough path is to generalize the concept of smooth path. We consider then a smooth path $X$, and without loss of generality we assume $X(0)=0$. Fix $t\in [0,1]$. In the definition of $k$-th signature we can replace the integral on $[0,1]$ with an integral on $[0,t]$. This is the same as restricting $X$ to $[0,t]$, hence we will denote such integral by $\sigma^{(k)}(X_{|[0,t]})$. %As an example,
% \[\sigma^{(1)}(X_{|[0,t]})_i=\int_{0}^{t}\dot{X_i}(\lambda)d\lambda=X_i(t) %-X_i(0)
% \]
For every $k$, we notice that $\sigma^{(k)}(X_{|[0,t]})$, as a function of $t$, is a path $[0,1]\to(\R^d)^{\otimes k}$. If we look at the full signature $\sigma(X_{|[0,t]})$, we get a path $[0,1]\to\G(\R^d)$. Notice that the signature of $X$ is the endpoint of such path. By \cite[Lemma 10]{G18}, this $\G(\R^d)$-valued path satisfies the H\"{o}lder-like inequality
\begin{equation}\label{eq:holder}
\left| \sigma^{(k)}(X_{|[s,t]})\right| \apprle |t-s|^k,
\end{equation}
where $f(x)\apprle g(x)$ means that there is a constant $c$ such that $f(x)\le c\cdot g(x)$ for every $x$.
Summing up, a smooth path $X\colon [0,1]\to\R^d$ induces a path $\sigma(X_{|[0,\cdot]})\colon [0,1]\to\G(\R^d)$ satisfying inequality (\ref{eq:holder}). If we want a rough path to be a generalization of a smooth path, we can use this property, also allowing different exponents. Let $p_k:T((\R^d))\to(\R^d)^{\otimes k}$ be the projection.

\begin{definition}
	A \emph{rough path} of order $m$ is a path ${\bf X}\colon [0,1]\to\G^m(\R^d)$ such that $|p_k({\bf X}(s)^{-1}\otimes {\bf X}(t))| \apprle |t-s|^\frac{k}{m}$
	for every $k\in\{1,\dots,m\}$ and every $s,t\in [0,1]$. The inverse is taken in the group $\G^m(\R^d)$.
\end{definition}

As we anticipated, we will focus on a special subclass of rough paths of order $m$, indexed by elements $L\in\Lie^m(\R^d)$.

\begin{definition}
	For $L\in\Lie^m(\R^d)$, consider the path ${\bf X}_L\colon [0,1]\to\G^m(\R^d)$ sending $t$ to $\exp^{(m)}(tL)$. %Define $${\bf X}_L=(p_1(g_L),p_2(g_L)).$$
	By \cite[Exercise 9.17]{FrizVictoir2010}, this is indeed an order $m$ rough path, and we define its \emph{signature} to be its endpoint $\sigma({\bf X}_L):=\exp(L)\in\G(\R^d)$.
\end{definition}

We want to parametrize the variety containing all $\sigma^{(k)}({\bf X}_L)$ when $L$ ranges in $\Lie^m(\R^d)$, so we are interested in the image of $p_k\circ\exp:\Lie^m(\R^d)\to (\R^d)^{\otimes k}$. While $\exp$ is not an algebraic map, $p_k\circ\exp$ is. In general, the image of an algebraic map is only a semialgebraic subset of the real affine space $(\R^d)^{\otimes k}$. However, it is simpler to work with complex projective varieties, so we will follow a common approach in applied algebraic geometry and consider the Zariski closure of the image, take the complexification and pass to the projectivization.

\begin{definition}\label{def:rough veronese}
	The \emph{rough Veronese variety} $\RR_{d,k,m}$ is the closure of the image of the composition
	\[f_{d,k,m}\colon \Lie^m(\R^d)\xrightarrow{\exp}\G(\R^d)\xrightarrow{p_k}(\R^{d})^{\otimes k} \rightarrow (\R^{d})^{\otimes k}\otimes\C=(\C^{d})^{\otimes k}\dashrightarrow\p^{d^k-1}
	.\]
\end{definition}
In general it is not easy to work out all the invariants of a given variety. Luckily, we will shortly see that the rough Veronese variety is toric. In other words it is the closure of the image of a monomial map. There are a lot of tools and techniques that make toric varieties accessible and easy to work with. Classical references on toric varieties are \cite{CLS11, BerndBook}. First, we fix some notation.
\begin{notation}\label{not:SR}
For every $d\ge 2$, let $W_{d,m}$ be the set of Lyndon words of length at most $m$ in the alphabet with $d$ letters. Let $S_{d,m}:=\C[x_w\mid w\in W_{d,m}]$ be a polynomial ring with as many variables as Lyndon words. On this algebra we define a grading by setting the weight of $x_w$ to be the length of $w$.

%Let $S=S_{m,d}=\C[y_\alpha]$ be a polynomial ring in variables $y_\alpha$ corresponding to Lyndon words $\alpha$ of length at most $m$ in alphabet with $d$ letters. The \emph{degree} of a variable corresponding to a Lyndon word $\alpha$ is the length of $\alpha$.

%Let $R=R_{d,k}=\C[x_1,\dots,x_{d^k}]$ be a polynomial ring with standard grading.

Let $A$ be the set of all monomials in $\C[x_w\mid w\in W_{d,m}]$ of weighted degree $k$. Define another polynomial ring $R_{d,k}:=\C[y_\alpha\mid \alpha\in A]$ with as many variables as elements of $A$, and the usual polynomial grading. 
\end{notation}
Our starting point will be the following result, proven in \cite[Proposition 19]{G18}. Not only we know that $\RR_{d,k,m}$ is toric, but we also know what are the monomials parametrizing it.
\begin{theorem}\label{thm:toricR}
Up to projectivization, $\RR_{d,k,m}$ is isomorphic to the closure of the image of the map $\Spec S_{d,m}\rightarrow\Spec R_{d,k}$ given by all monomials of weighted degree $k$. That is, the kernel of the map $R_{d,k}\rightarrow S_{d,m}$ is the homogeneous prime ideal defining $\RR_{d,k,m}$.
\end{theorem}

More precisely, it is the image of a weighted projective space by the map given by all sections of $\mathcal{O}(k)$, i.e.~all monomials of (weighted) degree $k$. Such a map does not have to be an embedding. In fact, $\mathcal{O}(k)$ does not need to be a Cartier divisor. 

\begin{remark}
	The classical $k$-Veronese variety is the image of the map given by all degree $k$ monomials in the usual grading. Theorem \ref{thm:toricR} shows that $\RR_{d,k,m}$ can be seen as a weighted version of the Veronese variety, thereby justifying its name.
\end{remark}
Theorem \ref{thm:toricR} allows us to investigate $\RR_{d,k,m}$ with the tools of toric geometry. In some of our combinatoric arguments, we will find it convenient to use the following notation.
\begin{notation}
	For any integers $a,b$ we denote by $(a^b)$ a sequence or multiset of $b$ elements equal to $a$. For example $(1^3,2^2)=(1,1,1,2,2)$.
\end{notation}

We start by describing the ideal of this variety, i.e.~the polynomials that vanish on it.

\subsection{Polynomials defining the rough Veronese variety}

The classical Veronese variety is defined by quadrics, so it is natural to ask whether $\RR_{d,k,m}$ enjoys the same property. Despite holding in many examples (see \cite[Section 4.3]{AFS18} and the computations in \cite{G18}), the conjecture that $\RR_{d,k,m}$ is always defined by quadrics is false. In \cite[Proposition 28]{G18} we see that the ideal $\RR_{2,20,14}$ has a cubic generator. One could still hope to bound the degree of the generators. 
Surprisingly no such bound exists, as we will shortly prove. First we need the following lemma from linear algebra.

\begin{lemma}\label{lem: magic square}
	For every $n\ge 3$, there exist $k_n\in\N$ and an $n\times n$ matrix $M_n$ with positive integer entries satisfying the following properties:
	\begin{enumerate}
		\item the sum of the entries of each row and each column of $M_n$ equals $k_n$,
		\item any subset of $n$ entries of $M_n$ that sum up to $k_n$ is either a row or a column; 
		%the only way to decompose the multiset of entries of $M_n$ into $n$-tuples summing to $k_n$ is by rows or by columns,
		\item All entries of $M_n$ are distinct; in particular, each column, is distinct from each row.
	\end{enumerate}
	We refer to the matrix $M_n$ as a \emph{rigid $k_n$-square matrix}.
	\begin{proof}
		We proceed by induction on $n$. If $n=3$, we take $k_3=31$ and
		\[M_3=\left( \begin{matrix}
		1 & 10 & 20\\
		16 & 6 & 9\\
		14 & 15 & 2
		\end{matrix}\right) .\]
		%		One can easily verify that the multiset $\{1,2,2,3,4,5,5,6,8\}$ has the desired property.		
		Assume the statement holds for $n$ and let us prove it for $n+1$. By induction hypothesis, there exist $k_n\in\N$ and an $n\times n$ matrix $M_n$ with the required properties. Let $\lambda\gg x\gg n$ be natural numbers and define $k_{n+1}=\lambda k_n+x$. Moreover, define the $(n+1)\times (n+1)$ matrices
		\[A=\left(
		\begin{tabular}{cccc|c}
		&&& & 0\\
		&$\lambda M_n$ && & 0\\
		&& && $\vdots$\\
		&&& & 0\\\hline\bigstrut
		0& $\dots$& 0& 0 & $\lambda k_n$
		\end{tabular}\right) \mbox{ and }\]\[
		B=\left(
		\begin{tabular}{cccccc|c}
		0&\dots& 0& 0&-2n+2&2n-3& $1+x$\\
		0&\dots&  0&-2n+3&0&2n-5& $2+x$\\
		$\vdots$&$\udots$&$\udots$&$\udots$&$\vdots$ & $\vdots$& $\vdots$\\
		0&-n-1 & 0& $\dots$&0&3& $n-2+x$\\
		-n&0&$\dots$ &$\dots$ &0&1& $n-1+x$\\
		0&0& $\dots$&$\dots$& 0&$-n^2$& $n^2+x$\\\hline\bigstrut
		$n+x$& $n+1+x$& $\dots$&$\dots$&$2n-2+x$& $2n-1+x$ & $nx-\frac{n(3n-1)}{2}$
		\end{tabular}\right).\]
		%		\begin{align}
		%		A_{n+1}&=\left(		\begin{tabular}{ccccc|c}
		%		&&& & & 0\\&&& & & 0\\
		%		&& $\lambda M_n$ && & $\vdots$\\
		%		&&& & & 0\\&&& & & 0\\\hline\bigstrut
		%		0&0& $\dots$& 0& 0 & $\lambda k_n$
		%		\end{tabular}\right) \mbox{ and}\nonumber\\
		%		B_{n+1}&=\left(
		%		\begin{tabular}{ccccc|c}
		%		0&$\dots$&$\dots$& 0& 0& n+1\\
		%		0&&\reflectbox{$\ddots$}& 0& 0& n+1\\
		%		$\vdots$&\reflectbox{$\ddots$}& \reflectbox{$\ddots$}& \reflectbox{$\ddots$} & $\vdots$& $\vdots$\\
		%		0&0&0& $\dots$& 0& n+1\\
		%		1&1&1& $\dots$& 1& 1\\\hline\bigstrut
		%		n&n& n& $\dots$& n & $-n^2+n+1$
		%		\end{tabular}\right).\nonumber		\end{align}
		Then we take $M_{n+1}:= A+B$. The only nontrivial property we have to show is point $2$. If a given subset $S$ of entries contains the lower right entry, then it may only contain the elements from the last row or column, as otherwise the sum of elements would exceed $k_{n+1}$. If it additionally contains $n^2+x$, then it must contain only the smallest elements left: that is those from the last column. If it does not contain $n^2+x$ then it must contain the largest elements left: those from the last row. 
		
		Hence, from now on we assume that $S$ does not contain the $(n+1,n+1)$ entry of $M_{n+1}$. As $\sum_{s\in S} s= x$ modulo $\lambda$ and $x\gg n$ we see that $S$ must contain precisely one element from last row or column.  As the sum of $n$ elements of the upper left $n\times n$ submatrix of $B$ is much larger then $-\lambda$ and much smaller than $\lambda$ we see that the corresponding elements of $A$ must sum to $\lambda k_n$. By inductive assumption $S$ must contain a row or a columns of $A$. Each such row or column can only be uniquely completed to a row or column of $M_{n+1}$, as all the entries in the last row and column are distinct. 
	\end{proof}
\end{lemma}
Thanks to Lemma \ref{lem: magic square}, we are now able to produce instances of $\RR_{d,k,m}$ generated in arbitrarily high degree.
\begin{proposition}\label{pro:rough veronese is not gen in bounded degree}
	For every $n\ge 3$ and every $d\geq 2$, there exist $k,m\in\N$ such that $k\ge m$ and the ideal of $\RR_{d,k,m}$ is not generated in degree $n$.
	\begin{proof}
		By Lemma \ref{lem: magic square}, there exists $k\in\N$ and a rigid $k$-square matrix $M$ of size $(n+1)\times(n+1)$. Let $m$ be the largest entry of $M$. %For every $d\ge 2$, let $W_{d,m}$ be the set of Lyndon words of length at most $m$ in the alphabet with $d$ letters. Let $\C[x_w\mid w\in W_{d,m}]$ be a polynomial algebra with as many variables as Lyndon words. On this algebra we define a grading by setting the weight of $x_w$ to be the length of $w$. 
		%		Now let $A$ be the set of all monomials in $\C[x_w\mid w\in W_{d,m}]$ of weighted degree $k$. Define another polynomial ring $\C[y_\alpha\mid \alpha\in A]$ with as many variables as elements of $A$, and the usual polynomial grading. 
		By Theorem \ref{thm:toricR}, the ideal $I$ of $\RR_{d,k,m}$ is the kernel of the ring homomorphism
		\[\f:\C[y_\alpha\mid \alpha\in A]\to\C[x_w\mid w\in W_{d,m}]\]
		sending $y_\alpha$ to the weighted monomial $\alpha$.
		
		For every $i\in\{1,\dots,m\}$, fix once and for all a variable of weight $i$. This means that we fix a length $i$ Lyndon word $w$ and we associate to $i$ the variable $x_w$. %We can take $d$ large enough to have as many Lyndon words of length $i$, and therefore as many weight $i$ variables, as we need.
%\todo[inline]{do we need that?}
		For every $j\in\{1,\dots,n+1\}$, consider the monomial $\alpha_j$ defined by the product of $n+1$ variables $x_w$ whose weights are the entries of the $j$-th row of $M$. By construction, this monomial has weighted degree $k$ and therefore $\alpha_1,\dots,\alpha_{n+1}\in A$. In a similar way, we can consider the monomials $\beta_1,\dots,\beta_{n+1}$ defined by the columns of $M$. We define
		\[f:=y_{\alpha_1}\cdot\ldots\cdot y_{\alpha_{n+1}}-y_{\beta_1}\cdot\ldots\cdot y_{\beta_{n+1}}.\]
		This is a degree $n+1$ element of $\C[y_\alpha\mid \alpha\in A]$. Moreover, since the degrees appearing in the rows of $M$ are the same as the degrees appearing in the columns, $f\in\ker(\f)=I$. Let us show that $f$ is not generated by degree $n$ elements. Let $q_1,\dots,q_r$ be the polynomials spanning the degree $n$ part of $I$. %that are of degree at most $n$. 
		Since $\RR_{d,k,m}$ is toric, we can assume these are binomials \cite[Proposition 1.1.9]{CLS11}. Hence, $q_i=g_i-h_i$, where each $g_i$ and each $h_i$ is a degree $n$ monomial. Suppose by contradiction that $f$ can be generated by $q_1,\dots,q_r$. Then there exist variables $l_1,\dots,l_r$ such that $f$ is the sum $$f=l_1(g_1-h_1)+\dots+l_r(g_r-h_r).$$ %It follows that each term in the sum is a multiple of some degree $n$ monomial. In particular, the monomial 
		Without loss of generality, assume that $y_{\alpha_1}\cdot\ldots\cdot y_{\alpha_{n+1}}=l_1g_1$. Since $l_1g_1-l_1h_1\in I$, we know that $\varphi(l_1h_1)=\varphi(l_1g_1)$. As a multiset, the variables in $\varphi(l_1h_1)$ are the entries of the matrix $M$. Furthermore, the variables of $l_1h_1$ provide a partition into $n+1$ sets of cardinality $n+1$, each with sum $k$. Since $h_1\neq g_1$, by assumption on $M$, this partition must correspond to the column partition of $M$. Hence, $l_1(h_1-g_1)=f$. This is a contradiction, because no variable divides $f$ by construction.
		%Since $q_1\in I$, the product of the monomials indexing the variables appearing in $g_1$ is the same as the one of $h_1$. Both $g_1$ and $h_1$ are product of $n$ variables, and each variable is indexed by a weighted monomial, that in turn is a way to write $k$ as a sum of its weights. So each of them gives $n$ ways to write $k$ as a sum of the weights of the variables appearing in $f$, and therefore it corresponds to $n$ rows or columns of $S$. Because of the rigidity of the square $S$, the only possibility is $g_1=h_1$, and that contradicts the minimality of $q_1,\dots,q_r$.
	\end{proof}
\end{proposition}

In applications, sometimes one has a signature coming from experiments and wants to know whether it belongs to a path of a certain class. In other words, we have a point in $\U_{d,k}$ and we want to understand if it belongs to a given signature subvariety. From this viewpoint Proposition \ref{pro:rough veronese is not gen in bounded degree} may seem disappointing. In order to check whether a given point belongs to $\RR_{d,k,m}$, we might have to evaluate polynomials of very high degree. The good news is that in most cases it is enough to only evaluate quadrics.  More precisely, $\RR_{d,k,m}$ is always defined by quadrics outside of a coordinate linear subspace of large codimension. In order to prove that, we recall a classical lemma from toric geometry.
\begin{lemma}\label{lem:toric generator}
	Let $s\in\N$ and let $I$ be the homogeneous toric ideal associated to a set of lattice points $A\subset\Z^n$. Let $g$ be a binomial which we write as $p_1+\ldots+p_t-(p_1'+\ldots+p_t')$, where $p_i,p_j'\in A$. Then $g\in I(s)$ if and only if starting from the multiset $\{p_i\}$ we can reach the multiset $\{p_j'\}$ in finitely many steps of the following form.
In each step we replace 	 $p_{i_1},\ldots,p_{i_s}$ with some other $r_{i_1},\ldots,r_{i_s}\in A$ such that $p_{i_1}+\ldots+p_{i_s}=r_{i_1}+\ldots+r_{i_s}$.
	
\end{lemma}
\begin{proof}
	A toric ideal is known to be binomial. %As $I$ is associated to a polytope, it is graded. 
	Thus, $f\in I(s)$ if and only if:
	$$f=\sum_{i=1}^k m_ib_i,$$
	where $m_i$ are monomials and $b_i\in I$ are binomials of degree at most $s$. The proof is by induction on $k$. The monomial of $f$ corresponding to $p_1+\dots+p_t$ must appear on the right hand side, say in $m_1b_1$. Then $f-m_1b_1$ is obtained by applying one step of the procedure described in the lemma. We conclude by induction.
\end{proof}

\begin{proposition}\label{prop:quad}
 	Fix $R=R_{d,k}$ as in Notation \ref{not:SR}. Let $I\subset R$ be the ideal of $\RR_{d,k,m}$ and let $I(2)\subset I$ be the vector subspace of degree two forms. Let $V\subset\p^{d^k-1}$ be the variety defined by $I(2)$. Then $\RR_{d,k,m}$ is an irreducible component of $V$.\end{proposition} As we will stress in Remark \ref{rmk:other components are small}, if there are other components, then they are contained in a small subspace of $\p^{d^k-1}$.
 	\begin{proof}
 		Since $I(2)\subset I$, it is clear that $\RR_{d,k,m}\subset V$. We want to find a polynomial $f\in R\setminus I$ such that the saturation of $I(2)$ with respect to $f$ is $I$. Recall that such saturation is defined by
 		\[(I(2):f^\infty):=\{g\in R\mid gf^k\in I(2)\mbox{ for some } k\in\N\}.
 		\]
 		So we want to prove that there exists an $f$ such that $(I(2):f^\infty)=I$. This implies that $I$ and $I(2)$ coincide in the localization $R_f$ and therefore $V$ and $\RR_{d,k,m}$ coincide in the affine open subset $(f\neq 0)$. In particular, $\RR_{d,k,m}$ is an irreducible component of $V$.
 		
 		Let $\mu=\dim\Lie^m(\R^d)$ and let $A\subset\N^\mu$ be the set of lattice points associated to the monomials defining the toric variety $\RR_{d,k,m}$. By Theorem \ref{thm:toricR}, points of $A$ are of the form
 		\[(w_{1,1},\ldots,w_{1,a_1},w_{2,1},\ldots,w_{2,a_2},\ldots,w_{m,1},\ldots,w_{m,a_m})\text{ where }\displaystyle{\sum_{i=1}^m\sum_{j=1}^{a_i} iw_{i,j}=k}.\]
 		%where $\displaystyle{\sum_{i=1}^m\sum_{j=1}^{a_i} iw_{i,j}=k}$. 
 		By \cite[Proposition 1.1.9]{CLS11}, there exists a set of binomial generators of $I$. Each of such binomials corresponds to an integral relation $p_1+\ldots+p_t=p_1'+\ldots+p_t'$, where all $p_i$ and $p_i'$ are in $A$ (the sum of points corresponds to the product of variables).
 		Our task is to find $f\notin I$ such that $f^ng\in I(2)$ for every generator $g=p_1+\ldots+p_t-p_1'-\ldots-p_t'$ of $I$ for some $n$. We will define $f$ to be a variable. In toric words, $f$ corresponds to a lattice point $p\in A$% of the form mentioned above
 		. By Lemma \ref{lem:toric generator}, we want to find $p\in A$ such that
 		$p+\ldots+p+p_1+\ldots+p_t$ can be turned in $p+\ldots+p+p_1'+\ldots+p_t'$ by repeatedly replacing \emph{a pair} of summands with another pair having the same sum. Let
 		\[f=p:=(k,0,\dots, 0)\] and $p_1=(w_{1,1},\ldots,w_{1,a_1},w_{2,1},\ldots,w_{2,a_2},\ldots,w_{m,1},\ldots,w_{m,a_m})$. Assuming $w_{2,1}>0$, we can replace $p+p_1$ by
 		\[(k-2,0,\dots,0,1,0,\dots,0)+(w_{1,1}+2,w_{1,2},\ldots,w_{1,a_1},w_{2,1}-1,w_{2,2}\dots,w_{2,a_2}%, \ldots,w_{m,1}
 		,\ldots,w_{m,a_m}),\]
 		i.e.~we replace two on the first coordinate with one on the coordinate corresponding to $w_{2,1}$.
 		%Notice we can do that by keeping both summands in $M$.
 		%Notice that the thesis is trivial when $k=1$, so we may safely assume $k-2\in\N$. 
 		By iterating this process, we add more copies of $p$ so that we can replace $p+\ldots+p+p_1$ by points for which \emph{at most one coordinate} $w_{i,j}$ is nonzero for $(i,j)\neq(1,1)$, and that coordinate is equal to one.
 		We do the same to $p_2,\dots,p_t$, so that we can replace $p+\ldots+p+p_1+\ldots+p_t$ by a sum of points for which at most one coordinate $w_{i,j}$ is nonzero for $(i,j)\neq(1,1)$, and that coordinate is equal to one. We apply the same process to $p_1'+\ldots+p_t'$. Since now both sides are broken into this kind of simple pieces, and since the condition of having weighted sum $k$ is always preserved, the only possibility is that the summands are pairwise the same.
 	\end{proof}
 
 \begin{remark}\label{rmk:other components are small}
In the proof of Proposition \ref{prop:quad} we could have chosen the point $p$ to be any point with coordinates $w_{i,j}=0$ for $i>1$. This proves that $I(2)$ and $I$ define the same scheme (in particular, the quadrics define the correct set) outside of the locus where \emph{all} these coordinates vanish. Therefore the other components of the variety defined by $I(2)$ (embedded or not) must be supported on a coordinate subspace of large codimension.
 \end{remark}

\subsection{Normality of the rough Veronese Variety}\label{sub:GeomRVV}
A classical approach to study geometry of a projective toric variety is to look at the associated lattice polytope \cite{CLS11, BerndBook, Fulton}. In our case the central object is presented in the following definition.

\begin{definition}\label{def:lattice polytope}For $k\in\N$ and $w=(w_1,\dots,w_r)\in\N^r$, we define the lattice polytope $P(w,k)$ as the convex hull of
	\[\left\lbrace (t_1,\dots,t_r)\in\N^r\ \left|\ \sum_{i=1}^{r}w_it_i=k\right. \right\rbrace .
	\]
\end{definition}
In this notation, the polytope associated to $\RR_{d,k,m}$ is $P((1^{a_1},2^{a_2},\dots,m^{a_m}),k)$, where $a_i$ is the number of Lyndon words of length $i$ in the alphabet $\{1,\dots,d\}$.
Many authors, like \cite{CLS11, Fulton}, % especially in theoretical mathematics,
 require a toric variety to be \emph{normal}. This is equivalent to the fact that the toric variety may be represented by a fan. Hence, the first important task is to decide when $\RR_{d,k,m}$ is normal. One advantage is that normality can be checked on the polytope.
\begin{lemma}\label{lem:easy facts about normality}
	A polytope $P$ is normal if and only if $kP$ is normal for every $k\in\N$. Moreover, if $X_P$ is the associated toric variety, then
	\begin{enumerate}
		\item $X_P$ is normal if and only if $P$ is very ample,
		\item $X_P$ is projectively normal (i.e.~the affine cone over it is normal) if and only if $P$ is normal.
	\end{enumerate}
\begin{proof}
	These results are proven in \cite[Chapter 2]{CLS11}, together with many other features of normal polytopes.
\end{proof}
\end{lemma}

Lemma \ref{lem:easy facts about normality} will help us to determine when $\RR_{d,k,m}$ is normal. We start with the following result, that simplifies the problem.

\begin{lemma}\label{lem:reduce to single variables}$P((w_1^{a_1},w_2^{a_2},\dots,w_m^{a_m}),k)$ is normal if and only if $P((w_1,w_2,\dots,w_m),k)$ is normal. %This means that adding or discarding repeated entries of $w$ does not affect the normality of $P(w,k)$. 
In particular, the normality of $\RR_{d,k,m}$ does not depend on $d$.
\begin{proof} %In order to simplify notations, set $P=P((w_1^{a_1},w_2^{a_2},\dots,w_m^{a_m}),k)$. 
By induction, it is enough to check what happens when we add or discard one of the entries. %Thus from now on we let 
In order to simplify notation, set $$P:=P((w_1,\dots,w_{i-1},w_i,w_i',w_{i+1},\dots,w_m),k),$$ where $w_i=w_i'$ and we do not assume that the $w_j$'s are distinct. %, therefore it is not restrictive to assume that $P$ has only one repeated entry $w_i$. 
%In order to distinguish the two repeated entries, we will call them $w_i$ and $w_i'$.	
Observe that one of the facets of $P$ is precisely the polytope $$Q:=P((w_1,\dots, w_i,w_{i+1},\dots,w_m),k).$$ Since every face of a normal polytope is normal, the first implication follows.

Assume now that $Q$ is normal. Consider the linear surjection $\sim:P\to Q$, defined by the sum of the two entries corresponding to $w_i$ and $w_i'$. If $t=(t_1,\ldots,t_{i-1},t_i,t_i',t_{i+1},\dots,t_m)\in P$, then $\tilde{t}=(t_1,\ldots,t_{i-1},t_i+t_i',t_{i+1},\dots,t_m)\in Q$. Note that $\sim$ is also surjective on lattice points, that is $\sim:P\cap\N^r\to Q\cap\N^{r-1}$ is surjective.  Let $p$ be a lattice point in some multiple $sP$ of $P$. We can write $p=\lambda_1p_1+\ldots+\lambda_sp_s$, where $p_1,\dots,p_s\in P$ are lattice points and $\lambda_1+\ldots+\lambda_s=s$. Since $\tilde{p}\in sQ$, $Q$ is normal by hypothesis and $\sim$ is surjective on lattice points, there are $x_1,\ldots,x_s\in P\cap\N^r$ such that $\tilde{p}=\lambda_1\tilde{p}_1+\ldots+\lambda_s\tilde{p}_s=\tilde{x}_1+\ldots+\tilde{x}_s$. Let $x_i\in P$ be the preimage of $\tilde{x}_i$ having 0 in the entry corresponding to $w_i'$, and let $x=x_1+\ldots+x_s$. Since $\tilde{p}=\tilde{x}$,  $p$ and $x$ coincide on every coordinate except the ones corresponding to $w_i$ and $w_i'$. Moreover, these entries have the same sum, and such sum is an integer by construction of $x$. Then it is enough to increase the zero entry of some suitable $x_i$, keeping such sum untouched, to obtain from $x_1,\dots, x_s$ another set of $s$ lattice points of $P$ whose sum is $p$. Therefore $P$ is normal.\end{proof}
\end{lemma}
In general, the rough Veronese variety does not need to be normal. As we show using the code available online \cite{CGM19}, $P ((1, 2, \ldots, 9), 18)$ is not very ample, so $\RR_{d,18,9}$ is not normal for any $d\geq 2$.

In algebraic geometry, given a map $f=(f_1,\dots,f_s)$ defined by homogeneous polynomials of the same degree, it is natural to study the induced \emph{rational map} $\tilde f$ between projective spaces. The map $\tilde f$ may be not defined everywhere - the locus where all of the polynomials $f_i$ vanish is called the \emph{base locus} or \emph{indeterminacy locus} of $\tilde f$. In our setting, the monomials defining $\RR_{d,k,m}$ are of degree $k$, however in variables that are of different degrees. This implies that instead of considering the classical projective space, we need the \emph{weighted projective space}. Theorem \ref{thm:toricR} tells us that $\RR_{d,k,m}$ is the closure of the image of a rational map
$$\p(1^{a_1},\dots,m^{a_m})\dashrightarrow\p^{d^k-1}.$$
The codomain is the usual projective space and the domain is the weighted projective space with $\sum_i a_i$ variables of weights as given in the brackets. The map is given by all monomials of degree $k$.

%We also define $h_{k,m}:\p(1,2,\ldots,m)\dashrightarrow\p^M$ to be the map defined by all monomials of weighted degree $k$. Then the image of $h_{k,m}$ is the toric variety associated to the polytope $P((1,2,\dots,m),k)$ of Definition \ref{def:lattice polytope}. By Lemma \ref{lem:reduce to single variables}, $\RR_{d,k,m}$ is normal if and only if the image of $h_{k,m}$ is normal. In particular, the normality of $\RR_{d,k,m}$ does not depend on $d$.

Before we proceed further, let us recall that maps from an algebraic variety $X$ to a projective space are studied through line bundles (or equivalently Cartier divisors) on $X$ \cite{Lazarsfeld}. The theory of line bundles on weighted projective spaces is very well understood. The Picard group equals $\Z$ where $\bf{1}$ is an ample generator. The global sections of this generator may be identified with all monomials of degree $l$ equal to the least common multiple of all the degrees appearing in the weighted projective space. Thus, it is customary to denote it by $\mathcal{O}(l)$. In particular, the elements of the Picard group will be denoted by $\mathcal{O}(s)$, for $s$ divisible by $l$.
Although the class group is abstractly also equal to $\Z$, it is larger. The inclusion of the Picard group in the class group $\Z\hookrightarrow\Z$ is given by multiplication by $l$. In particular, the elements of the class group will be denoted by $\mathcal{O}(s)$, for arbitrary $s$. 

\begin{lemma}
Let $\p(w_1,\dots,w_s)$ be any weighted projective space and let $l=\lcm(w_1,\dots,w_s)$. The map from $\p(w_1,\dots,w_s)$ to a projective space given by all monomials of degree $k$ does not have a base locus if and only if $l|k$.
\begin{proof}
	If $w_1\nmid l$, then $[1:0:\dots:0]$ is a base point. If all $w_i|l$ then the monomials $x_i^{\frac{l}{w_i}}$ do not have a base locus.
\end{proof}\end{lemma}

As there is some confusion in the literature, we stress that even if $l|k$ the map from the previous lemma does not have to be an embedding. In other words $\mathcal{O}(l)$, although always ample, may be not very ample.
\begin{example}
The polytope $P((1,6,10,15),30)$ is not very ample. Equivalently, the map $\p(1,6,10,15)\to \p^{17}$ given by all monomials of weighted degree 30 is not an embedding. This can be checked with the software \emph{Normaliz} \cite{Normaliz}, or by studying the polytope.
\end{example}
This raises the following question: when is $\RR_{d,k,m}$ an embedding of the weighted projective space $\p(1^{a_1},\dots, m^{a_m})$? 
We will prove that this happens if and only if $k$ is a multiple of all natural numbers between 1 and $m$. Equivalently, $\mathcal{O}(k)$ is a very ample line bundle if and only if $\lcm(2,\dots,m)\mid k$.
We will need the following technical lemma.
\begin{lemma}\label{lem:there are enough primes} If $m\ge 7$ and $m\neq 10$, then there exist two distinct prime numbers strictly larger than $\frac{m}{2}$ and at most equal to $m$ such that their sum is not a power of two.
	\begin{proof} If $m\le 56$, we can easily find the required primes.
		\[\begin{array}{|c|c|}
			\hline
			32\le m\le 56 & 29\mbox{ and }31\\
			\hline			20 \le m \le 31& 17\mbox{ and }19\\
				\hline		14 \le m \le 19& 11\mbox{ and }13\\
				\hline		11\le m \le 13& 7\mbox{ and }11\\
			\hline			m\in\{7,8,9\}& 5\mbox{ and }7\\\hline
		\end{array}\]
%		\begin{itemize}
%			\item For $32\le m\le 56$, take the prime numbers 29 and 31,
%			\item for $20 \le m \le 31$, take the prime numbers 17 and 19,
%			\item for $14 \le m \le 19$, take the prime numbers 11 and 13,
%			\item for $11\le m \le 13$, take the prime numbers 7 and 11,
%			\item for $m\in\{7,8,9\}$, take the prime numbers 5 and 7 %,and for $m=5$, take 3 and 5.		\end{itemize}
	Assume then $m\ge 57$. For $x\in\N$, let $\pi(x)$ be the number of prime numbers smaller or equal than $x$. In \cite{RS62}, the authors show that
		$$\frac{x}{\log x} < \pi(x) < 1.3 \frac{x}{\log x}$$
		for every $x\geq 17$. Now,
		\[
		\pi(m)-\pi\left(\frac{m}{2}\right)\ge \frac{m}{\log m} - 1.3\frac{m}{2 \log \left(\frac{m}{2}\right)},\] %= m\cdot\frac{\log m-\log 2-0.65\log m}{\log m (\log m - \log 2)},
%		\]
		and the latter is at least 3 for $m\geq 57$, so there are at least three prime numbers $a<b<c$ between $\frac{m}{2}$ and $m$. If $b+c$ is not a power of 2, we are done. Otherwise we have $b+c=2^n$. Then
		\[2^{n-1}=\frac{b+c}{2}\le m<a+b<b+c=2^n,\]
		so $a+b$ cannot be a power of 2.
	\end{proof}
\end{lemma}

\begin{theorem}\label{thm: normality of rough Veronese}
If $i\mid k$ for every $i\in\{1,\dots,m\}$, then $\RR_{d,k,m}$ is projectively normal.
\begin{proof} For $m\le 6$, normality can be explicitly checked with the software \emph{Polymake} \cite{polymake}, so we assume $m\ge 7$.
By Lemmas \ref{lem:easy facts about normality} and \ref{lem:reduce to single variables}, we may assume that $k=\lcm(1,\dots,m)$ and study the map
\[\p(1,2,\ldots,m)\to\p^N\]
defined by all monomials of weighted degree $k$. Let $P:=P((1,2,\ldots,m),k)$ be the associated polytope. For $s\in\N$, consider the dilation $sP$ of $P$. A lattice point of $sP$ is of the form $(a_1,\dots,a_m)\in\N^m$ with $a_1+2a_2+\ldots+ma_m=sk$. We can write it as a multiset
\[M=\{1^{a_1},\dots,m^{a_m}\}.\]
By induction on $s$, we only need to prove that there exists a submultiset $S$ of $M$ whose entries sum up to $k$. Let us modify $M$ to $M'$ in the following way. If there are $i,j\in M$ such that $i,j\le \frac{m}{2}$, then discard $i$ and $j$ and add $i+j$. Notice that $i+j\le m$. %and so, the least common multiple of the entries of $M'$ is also $k$. Now, i
If $S'$ is a submultiset of $M'$ whose entries sum up to $k$, then either $S$ does not contain $i+j$, and therefore $S=S'$ is a submultiset of $M$ as well, or $S'$ contains $i+j$, and we define $S$ by replacing back $i+j$ with $i,j$. By iterating this argument, we can assume that there is at most one $i\in\{1,\dots,m\}$ such that $i\le\frac{m}{2}$.

%Iterate this step until there is at most one entry in $M'$ smaller than $\frac{m}{2}$. 
The resulting multiset $M'$ is of the form $\left\{i_1^{b_{i_1}},\dots, i_t^{b_{i_t}} \right\}$. Let $k'=\lcm(i_1,\dots,i_t)$ and consider the map
\[\p(i_1,i_2,\ldots,i_t)\to\p^T\]
given by all the monomials of weighted degree $k'$. Its image is associated to the polytope $P':=P((i_1,i_2,\ldots,i_t),k')$.  By construction,  $k'$ divides $k$ and $\dim P'\le t-1 \le \frac{m}{2}$. Observe that any point of $P'$ has entries summing up to $k'$. Since the entries of $M'$ sum up to $sk$, $M'$ is a point in $rP'$ for $r = \frac{sk}{k'}$. In fact, $M'\in s\left(\frac{k}{k'}P'\right)$ and so, it is enough to prove that $\frac{k}{k'}P'$ is normal. This will imply that $M'$ is a sum of lattice points of $\frac{k}{k'}P'$, and therefore it admits the required submultiset $S'$. 

If there exists a prime number $\frac{m}{2}\le p \le m$ such that $p\nmid k'$, then $p\mid \frac{k}{k'}$ and so $\frac{k}{k'} \ge p \ge \frac{m}{2} \ge \dim P'$. In this case, $\frac{k}{k'}P'$ is normal by \cite[Theorem 2.2.12]{CLS11}. This guarantees the existence of the desired $S'$. 
Otherwise, assume $m\neq 10$. By Lemma \ref{lem:there are enough primes}, there exist two prime numbers $p_1, p_2$ between $\frac{m}{2}$ and $m$ such that $p_1+p_2$ is not a power of $2$. Up to relabelling, we can assume that $b_{p_1}\le b_{p_2}$. 

Once more, we have to modify $M'$ by making some replacements. For $b_{p_1}$ times, discard an entry $p_1$ and an entry $p_2$ and add an entry $p_1+p_2$. Denote by $M''$ the resulting multiset and by $P''$ the associated polytope, that satisfies $\dim P'' \le \frac{m}{2}$. As before, it is enough to find a submultiset $S''$ of $M''$ such that the entries of $S''$ sum up to $k$. By construction, the entries of $M''$ still sum up to $sk$ and their least common multiple $k''$ divides $k$. Indeed, the sum $p_1+p_2$ is even and not a power of two, thus all powers of prime numbers that divide it are smaller or equal to $m$. Notice that neither $p_1$, nor any of its multiples appear in $M''$. Therefore, $p_1 \mid k$ and $p_1\nmid k''$, which implies that  $p_1\mid \frac{k}{k''}$ and so $\frac{k}{k''} \ge p_1 \ge \frac{m}{2} \ge \dim P''$. By \cite[Theorem 2.2.12]{CLS11}, $\frac{k}{k''}P''$ is normal. This guarantees the existence of the desired $S''$.

The last case $m=10$ can be explicitly checked with the software Normaliz.
\end{proof}
\end{theorem}

\begin{corollary} 
	The map $\p(1^{a_1},\dots,m^{a_m})\to\p^N$ given by all monomials of weighted degree $k$ is an embedding if and only if $i\mid k$ for every $i\in\{1,\dots,m\}$. If so, $\RR_{d,k,m}\simeq\p(1^{a_1},\dots,m^{a_m})$ is embedded as a projectively normal variety.
	\begin{proof} By Theroem \ref{thm: normality of rough Veronese}, the polytope associated to $\RR_{d,k,m}$ is normal and thus very ample. This implies that the map is an embedding and the image is projective normal.
	\end{proof}
\end{corollary}

\subsection{Dimension and Degree of the rough Veronese variety}\label{sub:degRVV}
 
In this section we provide formulas for the dimension and degree of $\RR_{d,k,m}$. We work under the assumption that $k\geq m$, as otherwise some variables do not appear at all in the parametrization and hence we can easily reduce to this case. Recall that the \emph{normalized volume} of a polytope $P\subset\R^n$, denoted by $\vol P$, is $n!$ times its Lebesgue measure.

\begin{proposition}\label{pro: dim and deg of rough veronese} Let $d,k,m\in\N$ with $d\ge 2$ and $k\ge m$.
	Let $W_{d,m}$ be the set of Lyndon words of length at most $m$ in the alphabet with $d$ letters.
	\begin{enumerate}
		\item The dimension of $\RR_{d,k,m}$ is $\# (W_{d,m})-1$.
		\item Let $l=\lcm(1,2,\dots,m)$ and let $\Delta$ be the convex hull of the integral points
		$$\left(\frac{l}{1},0,\dots,0\right),\left( 0,\frac{l}{1},0,\dots,0\right),\dots,\left(0,\dots,0,\frac{l}{m}\right),$$
		where the number of occurrences of $\frac{l}{i}$ is the number of Lyndon words of length $i$. Then
		\begin{equation}\label{equat:degree of rough Veronese}
		\deg\RR_{d,k,m}\le\vol\left( \frac{k}{l}\Delta\right)
		\end{equation}
		and
		\[\lim_{k\rightarrow\infty}\frac{\deg \RR_{d,k,m}}{ \vol \left(\frac{k}{l}\Delta\right) }=1.\]
		Moreover, equality holds in (\ref{equat:degree of rough Veronese}) if and only if $l|k$.
	\end{enumerate}
\end{proposition}
The dimension of $\RR_{d,k,m}$ was already computed in \cite[Remark 6.5]{AFS18} by proving that the fibers of the map $f_{d,k,m}$ have dimension 0. However, in our opinion it is interesting to have a different proof based on toric techniques.
\begin{proof} Let $a_i$ be the number of Lyndon words of length $i$. We know that $\RR_{d,k,m}$ is the toric variety associated to the polytope
	$$P:=P((1^{a_1},\dots, m^{a_m}),k).$$
\begin{enumerate}
	\item Note that $\dim\RR_{d,k,m}\le\dim\Lie^m(\R^d)-1=\# (W_{d,m})-1$, as this is the dimension of the parameterizing weighted projective space. Thus, in order to prove the statement we only have to show that $P$ has maximal possible dimension. In fact, we prove that the lattice points of $P$ generate the lattice
	$$L=\left\{(x_{1,1},\dots,x_{1,a_1},x_{2,1},\dots,x_{m,a_m})\in\Z^{\#W_{d,m}}\mid k\mbox{ divides }\sum_{i=1}^m\sum_{j=1}^{a_i}ix_{i,j}\right\}.$$
	If $i\in\{1,\ldots,m\}$, $j\in\{1,\ldots,a_i\}$ and $(i,j)\neq (1,1)$, then there exists a lattice point $p\in P$ with $x_{i,j}=1$ and $x_{i',j'}=0$, unless $(i',j')=(i,j)$ or $(i',j')=(1,1)$. Take any $z\in L$. By adding and subtracting $p$ of the above type, we can reduce to the situation when $z$ has only the first coordinate nonzero. In such a case, the claim is obvious.
	\item We note that $P$ is the convex hull of lattice points in $\frac{k}{l}\Delta$. Thus,
	\[\deg\RR_{d,k,m}=\vol P\leq \vol \frac{k}{l}\Delta.\]
	Now, equality holds if and only if $P= \frac{k}{l}\Delta$. However, the latter is a lattice polytope if and only if $l|k$. Finally, note that $P((1^{a_1},\dots, m^{a_m}),k)$ is injected into $P((1^{a_1},\dots, m^{a_m}),k+1)$ simply by adding $(1,0,\dots,0)$. Hence
	$$\vol \left\lfloor\frac{k}{l}\right\rfloor\Delta\leq \vol P=\deg \RR_{d,k,m}\leq \vol \left\lceil\frac{k}{l}\right\rceil\Delta, $$
	which finishes the proof.\qedhere
\end{enumerate}\end{proof}

Many software-aided computations on the dimension, the degree and even the generators of the ideal of $\RR_{d,k,m}$ are presented in \cite{G18}. We now want to use Proposition \ref{pro: dim and deg of rough veronese} to deal with this numbers by using toric geometry.

\begin{example} Fix $d=m=2$. Since $W_{2,2}=\{1,2,12\}$, $\dim\RR_{2,k,2}=2$. We now compute the degree for small values of $k$. Since $\RR_{2,k,2}$ is a surface, we will deal with 2-dimensional polytopes. Here, $l=\lcm(1,2)=2$ and we denote $P:=P((1^2,2),k)$.
	\begin{enumerate}
		\item[\emph{($k=2$)}] As in Notation \ref{not:SR}, $R_{2,2}=\C[x^2,xy,y^2,a]$ and $P=\Delta$ is the convex hull of the points $(2,0,0),(0,2,0),(0,0,1)$. It is a triangle
		
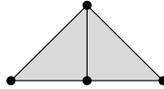
\[\begin{tikzpicture}
%		\draw[step=1cm,gray,very thin] (-0.1,-0.1) grid (2.1,1.1);
		\filldraw[fill=black!15!white, draw=black] (0,0) -- (2,0) -- (1,1) -- cycle;
		\draw (1,0) -- (1,1);
		\fill[black] (0,0) circle (0.06cm);
		\fill[black] (1,0) circle (0.06cm);
		\fill[black] (2,0) circle (0.06cm);
		\fill[black] (1,1) circle (0.06cm);
		\end{tikzpicture}\]
		of normalized area 2, hence $\deg \RR_{2,2,2}=2$. As \cite[Example 15]{G18} shows, $\RR_{2,2,2}$ is the cone over a smooth conic in $\p^3$.
		\item[\emph{($k=3$)}] Since $R_{2,3}=\C[x^3,x^2y,xy^2,y^3,xa,ya]$, $\frac{3}{2}\Delta=\conv\left\lbrace (3,0,0),(0,3,0),(0,0,\frac{3}{2})\right\rbrace $, while $P=\conv\{(3,0,0),(0,3,0),(1,0,1),(0,1,1)\}$.
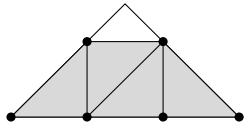
		\[\begin{tikzpicture}
		\draw (0,0) -- (3,0) -- (1.5,1.5) -- cycle;
		\filldraw[fill=black!15!white, draw=black] (0,0) -- (3,0) -- (2,1) -- (1,1) -- cycle;
		\draw (1,1) -- (1,0) -- (2,1) -- (2,0);
		\fill[black] (0,0) circle (0.06cm);
		\fill[black] (1,0) circle (0.06cm);
		\fill[black] (2,0) circle (0.06cm);
		\fill[black] (3,0) circle (0.06cm);
		\fill[black] (1,1) circle (0.06cm);
		\fill[black] (2,1) circle (0.06cm);
		\end{tikzpicture}\]
		Therefore, $\deg\RR_{2,3,2}=\vol P=4<\frac{9}{2}=\vol \frac{3}{2}\Delta$.
		\item[\emph{($k=4$)}] In this case,  $R_{2,4}=\C[x^4,x^3y,x^2y^2,xy^3,y^4,x^2a,xya,y^2a,a^2]$, and so, $P=2\Delta=\conv\{(4,0,0),(0,4,0),(0,0,2)\}$ and $\deg \RR_{2,4,2}=\vol P=8$.
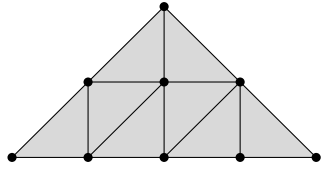
		\[\begin{tikzpicture}
		\filldraw[fill=black!15!white, draw=black] (0,0) -- (4,0) -- (2,2) -- cycle;
		\draw (3,1) -- (1,1) -- (1,0) -- (2,1) -- (2,0) -- (3,1) -- (3,0);
		\draw (2,1) -- (2,2);
		\fill[black] (0,0) circle (0.06cm);
		\fill[black] (1,0) circle (0.06cm);
		\fill[black] (2,0) circle (0.06cm);
		\fill[black] (3,0) circle (0.06cm);
		\fill[black] (4,0) circle (0.06cm);
		\fill[black] (1,1) circle (0.06cm);
		\fill[black] (2,1) circle (0.06cm);
		\fill[black] (3,1) circle (0.06cm);
		\fill[black] (2,2) circle (0.06cm);
		\end{tikzpicture}\]
		
		\item[\emph{($k=5$)}] As in the case $k=3$,  $P\subsetneq\frac{5}{2}\Delta$. More precisely, $\frac{5}{2}\Delta=\conv\left\lbrace (5,0,0),(0,5,0),(0,0,\frac{5}{2})\right\rbrace $, and $P=\conv\{(5,0,0),(0,5,0),(1,0,2),(0,1,2)\}$.
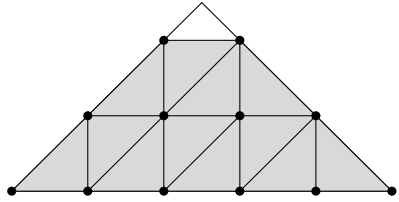
		\[\begin{tikzpicture}
		\draw (0,0) -- (5,0) -- (2.5,2.5) -- cycle;
		\filldraw[fill=black!15!white, draw=black] (0,0) -- (5,0) -- (3,2) -- (2,2) -- cycle;
		\draw (1,1) -- (1,0) -- (2,1) -- (2,0)  -- (3,1) -- (3,0) -- (4,1) -- (4,0);
		\draw (1,1) -- (4,1);
		\draw (2,2) -- (2,1) -- (3,2) -- (3,1);
		\fill[black] (0,0) circle (0.06cm);
		\fill[black] (1,0) circle (0.06cm);
		\fill[black] (2,0) circle (0.06cm);
		\fill[black] (3,0) circle (0.06cm);
		\fill[black] (4,0) circle (0.06cm);
		\fill[black] (5,0) circle (0.06cm);
		\fill[black] (1,1) circle (0.06cm);
		\fill[black] (2,1) circle (0.06cm);
		\fill[black] (3,1) circle (0.06cm);
		\fill[black] (4,1) circle (0.06cm);
		\fill[black] (2,2) circle (0.06cm);
		\fill[black] (3,2) circle (0.06cm);
		\end{tikzpicture}\]
		Hence, $\deg\RR_{2,5,2}=\vol P=12<\frac{25}{2}=\vol \frac{5}{2}\Delta$.
	\end{enumerate}
In this way, it is straightforward to see that, up to a linear isometry, the simplex $\frac{k}{2}\Delta$ is the triangle with vertices $(k,0,0),(0,k,0)$ and $\left( 0,0,\frac{k}{2}\right)$. If $k$ is even, then they are lattice points, $P=\frac{k}{2}\Delta$, and
\[\deg\RR_{2,k,2}=\vol P=\vol \frac{k}{2}\Delta=\frac{k^2}{2}.\]
If $k$ is odd, then $P$ is the trapezium with vertices
$(k,0,0),(0,k,0),\left( 1,0,\frac{k-1}{2}\right)$ and $\left(0 ,1,\frac{k-1}{2}\right)$, implying that
\[\deg\RR_{2,k,2}=\vol P=\frac{k^2-1}{2}<\frac{k^2}{2}=\vol \frac{k}{2}\Delta.\]
\end{example}

\section{Axis-parallel paths}\label{sec:axis parallel}
Besides $\RR_{d,k,m}$, the universal variety contains other interesting subvarieties. One of them is the signature variety $\LL_{d,k,m}$ of piecewise linear paths. Up to translation, a piecewise linear path $X$ with $m$ steps is the concatenation of $m$ linear paths, each represented by a vector $v_i$. This decomposition is unique, provided that $v_{i+1}$ is not a multiple of $v_i$ for any $i$. In this section, we study the subfamily of \emph{axis-parallel paths}.  
\begin{definition} Let $\{e_1,\dots, e_d\}$ be the standard basis for $\mathbb{R}^d$. A piecewise linear path $X=v_1\sqcup\ldots\sqcup v_m$ is an axis-parallel path (or simply an axis path) if there are $a_1,\dots,a_m\in\R$ such that $v_i=a_ie_{\nu_i}$ for every $i$, where $\nu_i\in \{1,\dots,d\}$.
\end{definition}
In other words, each step is a multiple of a basis vector. Therefore, an axis-parallel path is characterized by two sequences. One of them is the sequence $\nu=(\nu_1,\dots,\nu_m)$, called \emph{shape} of $X$, that stores in each $\nu_i\in \{1,2,\dots,d\}$ the direction of the $i$-th step. The other sequence, $a =(a_1,\dots,a_m)\in\R^m$, stores the length of each step and we call it the \emph{sequence of lengths}. Notice that $\ell(\nu)=\ell(a)=m$.
When we study an axis-parallel path $X\colon [0,1]\to\R^d$, we may assume that the image is not contained in any hyperplane, that is, it is nondegenerate. This means that $\{\nu_1,\dots,\nu_m\}=\{1,\dots, d\}$. %Moreover, our assumption that each $v_{i+1}$ is not a multiple of $v_i$ guarantees that $\nu$ has no consecutive repeated entries.

The $k$-th signature of an axis-parallel path $X$ can be computed combinatorially in a very nice way. Recall that a \emph{partition} of a set $S$ is a collection of subsets (called blocks) such that their union is $S$. Now each sequence $\nu$ induces a partition $\pi_\nu=\{\pi_1|\pi_2|\dots|\pi_d\}$ of the set $\{1,\dots, m\}$, defined by $(\pi_\nu)_i=\{j\in\{1,\dots,m\}\mid \nu_j=i\}$.
For instance, if $\nu=(1,2,1,3,3,1)$ then $\pi_\nu = \{1,3,6|2|4,5\}$. We will write $\pi$ instead of $\pi_\nu$ when there are no ambiguities.
%\begin{remark}A set partition of $S$ is a collection of disjoint subsets, called blocks, of $S$ whose union is $S$. The number of set partitions of the set $\{1,\dots,n\}$ is the Bell number $B_n$. \end{remark}
We can now introduce the main character of this section.
\begin{definition}
	Fix $\nu=(\nu_1,\dots,\nu_m)$ and $k\in\N$. Let $d=\max(\nu_1,\ldots,\nu_m)$. For an axis parallel path $X=a_1e_{\nu_1}\sqcup\ldots\sqcup a_me_{\nu_m}$, let $g_{\nu,k}(X):=\sigma^{(k)}(X)$ be its $k$-th signature. As we did in Definition \ref{def:rough veronese}, we pass to the complex projective space and we define the \emph{axis paths variety} $\mathcal{A}_{\nu,k}$ to be closure of the image of the composition \[\R^m\xrightarrow{g_{\nu,k}}(\R^d)^{\otimes k}\rightarrow(\R^d)^{\otimes k}\otimes\C=(\C^d)^{\otimes k}\dashrightarrow\p^{d^k-1}.\]
	Instead of $g_{\nu,k}$, we can also consider 
	\[G_{\nu,k}:\R^m\to\R^d\times (\R^d)^{\otimes 2}\times\ldots\times (\R^d)^{\otimes k}\] by sending $a\mapsto (g_{\nu,1}(a),\dots,g_{\nu,k}(a))$. In this case we denote by $\AAA_{\nu,\le k}\subset \C^d\times(\C^d)^{\otimes 2}\times\ldots\times (\C^d)^{\otimes k}$ the closure of the complexification of the image of $G_{\nu,k}$.
\end{definition}

There is a nice way to write down the polynomials defining the map $g_{\nu,k}$. The following result is a consequence of \cite[Corollary 5.3]{AFS18}.
\begin{lemma}\label{lem:CombDescription}
Let $X$ be the axis-parallel path of shape $\nu$ and sequence of lengths $a$. Then the $(i_1\dots i_k)$-th entry of the $k$-th signature is
\begin{eqnarray*}
\sigma(X)_{i_1\dots i_k} = \sum_{(j_1,\dots,j_k)} \frac{1}{s_1!s_2!\cdots s_m!} a_{j_1}a_{j_2}\cdots a_{j_k},
\end{eqnarray*}
where we sum over all the non-decreasing sequences $(j_1,j_2,\dots,j_k)$ such that $j_l\in \pi_{i_l}$ for $l\in\{1,\dots, k\}$, and $s_l$ counts the number of times that $l$ appears in $(j_1,j_2,\dots,j_k)$.
\end{lemma}
Lemma \ref{lem:CombDescription} allows us to find the $k$-th signature of an axis-parallel path and to explicitly write the map $g_{\nu,k}$. We implemented a Macaulay2 code to compute the ideal of $\AAA_{\nu,k}$. This code can be found in \cite{CGM19}.
\begin{example}
Let $\sigma$ be the 4-th signature of an axis path of shape $\nu= (1,2,1,3,2,3,1,4)$ and sequence of lengths $a=(a_1,\dots, a_8)$. Then
$\pi_\nu = \{1,3,7|2,5|4,6|8\}$. By Lemma \ref{lem:CombDescription},
\begin{eqnarray*}
\sigma_{1234} &=& a_1a_2a_4a_8 + a_1a_2a_6a_8 + a_1a_5a_6a_8 + 
a_3a_5a_6a_8, \\
\sigma_{2314} &=& a_2a_4a_7a_8 + a_2a_6a_7a_8 + a_5a_6a_7a_8, \\
\sigma_{4123} &=& 0, \\
\sigma_{1124} &=& \frac{1}{2}a_1^2a_2a_8 + \frac{1}{2}a_1^2a_5a_8 + a_1a_3a_5a_8 + \frac{1}{2}a_3^2a_5a_8.
\end{eqnarray*}
\end{example}
%Now, fix a sequence $\nu$ and a nonnegative integer $k$, and so $d=\max(\nu_,\ldots,\nu_m)$ and $m=\ell(\nu)$. We look at the variety given by the $k$-th signature of the axis-parallel paths of shape $\nu$. whose entrywise description is given by Lemma \ref{lem:CombDescription}. 

\begin{remark}\label{rmk:fill the universal}
	As any signature variety, the axis paths variety is contained in the universal variety. Namely, $\AAA_{\nu,k}\subset\U_{d,k}$ for every $d\geq \max(\nu_1,\ldots,\nu_m)$. It is known that $\LL_{d,k,m}$ coincides with $\U_{d,k}$ for $m$ large enough, because every piecewise smooth path can be approximated by using piecewise linear paths. Since every piecewise linear path can be approximated by an axis path, we deduce that for every $d$ and $k$, there exists $\nu$ such that $\max(\nu_i)=d$ and $\AAA_{\nu,k}=\U_{d,k}$.
\end{remark}

It is worth to write down what Chen's identity means for the axis paths signature variety.

\begin{proposition}\label{cor:Chen for axis paths}
	Consider two sequences $\nu_1$ and $\nu_2$. If $X_1$ and $X_2$ are axis-parallel paths of shape $\nu_1$ and $\nu_2$ respectively, then
 	\[g_{\nu_1\nu_2,k}(X_1\sqcup X_2)=\sum_{a+b= k}g_{\nu_1,a}(X_1)\otimes g_{\nu_2,b}(X_2).\]
	In particular, $\AAA_{\nu_1\nu_2,\le k}$ is a projection of the Segre product $\AAA_{\nu_1,\le k}\times\AAA_{\nu_2,\le k}$.	
\end{proposition}

In Section \ref{sec:rough veronese} we saw how important it is for a variety to be toric, so it is natural to ask whether $\AAA_{\nu,k}$ enjoys such property.

\subsection{Toricness of $\AAA_{\nu,k}$}\label{sub: axis toricness}
It is not difficult to find instances when the axis paths variety is toric. As an example we can take $\mathcal{A}_{\nu,1} = \mathbb{R}^{d}$, or consider $\nu=(1,2,\dots,d)$ and obtain the Veronese variety $\V_{d,k}=\AAA_{\nu,k}$. A less trivial example is given by Remark \ref{rmk:fill the universal}. Thanks to Lemma \ref{lem:CombDescription}, we are able to perform efficient computations and check that many other occurrences of $\AAA_{\nu,k}$ are toric, and even find their polytopes.

\begin{example}
	In \cite{CGM19} we compute the ideal of $\AAA_{(1,2,1),3}$ and we check that it is toric. Further computations on the ideal allow us to determine its polytope $P$. Indeed, $\AAA_{(1,2,1),3}$ is a degree 6 surface in $\p^6$, so $P$ is a two-dimensional lattice polytope of normalized area 6 that contains exactly 7 lattice points. By checking the Hilbert and Erhart polynomials and the Betti table, we see that there are only two possible polytopes.
	\begin{center}\scalebox{0.7}
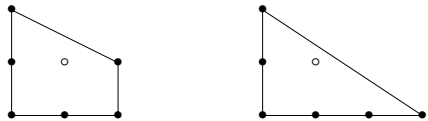
{
			\begin{tikzpicture}
			\draw (0,0) -- (2,0) -- (2,1) -- (0,2) -- (0,0);
			\node at (0,0) {$\bullet$};
			\node at (1,0) {$\bullet$};
			\node at (2,0) {$\bullet$};
			\node at (2,1) {$\bullet$};
			\node at (0,1) {$\bullet$};
			\node at (0,2) {$\bullet$};
			\node at (1,1) {$\circ$};
			\end{tikzpicture} \hspace{2cm}
			\begin{tikzpicture}
			\draw (0,0) -- (3,0) -- (0,2) -- (0,0);
			\node at (0,0) {$\bullet$};
			\node at (1,0) {$\bullet$};
			\node at (2,0) {$\bullet$};
			\node at (3,0) {$\bullet$};
			\node at (0,1) {$\bullet$};
			\node at (0,2) {$\bullet$};
			\node at (1,1) {$\circ$};
			\end{tikzpicture}}
	\end{center}
	By looking at the intersections of the tangent spaces at singular points with the variety, we see that the correct one is the one on the right.\end{example}
	
In \cite{G18} the second author presents a change of coordinates, based on the exponential description of the tensor signature, for which the universal variety $\mathcal{U}_{d,k}$ is the image of a monomial map, and therefore it is toric. However, such change of coordinates fails to prove that every $\AAA_{\nu,k}$ is toric as well. For instance, it does not turn the ideal of $\AAA_{(1,2,1),3}$ into a binomial ideal. One could still hope to find another change of coordinates that makes both the ideals of $\mathcal{A}_{\nu,k}$ and $\mathcal{U}_{d,k}$ binomial. Unfortunately, this is not true.

\begin{example}
	 %For this example it is convenient to work in the affine space $\mathbb{C}^8$.
	The polytopes $P$ and $Q$ of  $\mathcal{U}_{3,3} %= \mathcal{A}_{(1,2,1,2,1), 3}
	$ and $\mathcal{A}_{(1,2,1), 3}$ are
	\begin{center}
		\usetikzlibrary{arrows}\scalebox{0.7}
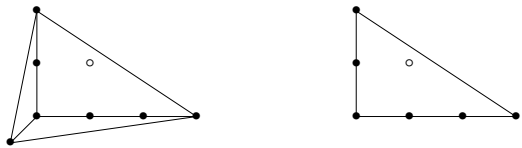
{
			\begin{tikzpicture}
			\draw (0,0) -- (3,0) -- (0,2) -- (0,0);
			\draw (0,0) -- (-0.5,-0.5) -- (0,2);
			\draw (-0.5,-0.5) -- (3,0);
			\node at (-0.5,-0.5) {\textcolor{black}{$\bullet$}};
			\node at (0,0) {$\bullet$};
			\node at (1,0) {$\bullet$};
			\node at (2,0) {$\bullet$};
			\node at (3,0) {$\bullet$};
			\node at (0,1) {$\bullet$};
			\node at (0,2) {\textcolor{black}{$\bullet$}};
			\node at (1,1) {$\circ$};
			%\draw[thick, ->, >=stealth'] (4,1) -- (5,1);
	%		\node at (4.5,1.25) %{$\varphi$};
			\draw (6,0) -- (9,0) -- (6,2) -- (6,0);
			\node at (6,0) {$\bullet$};
			\node at (7,0) {$\bullet$};
			\node at (8,0) {$\bullet$};
			\node at (9,0) {$\bullet$};
			\node at (6,1) {$\bullet$};
			\node at (6,2) {$\bullet$};
			\node at (7,1) {$\circ$};
			\end{tikzpicture}}
	\end{center}%The toric inclusions $\mathcal{A}_{(1,2,1), 3}\subset\mathcal{U}_{3,3}\subset\p^7$ would induce surjections %the inclusions $T_X \subset T_Y \subset T_{\mathbb{C}^8}$ and the diagrams
%	\begin{eqnarray*}
%		\xymatrix{  & \mathbb{Z}^8\ar@{->>}[ld]\ar@{->>}[rd]  & \\ \mathbb{Z}^2 & & \mathbb{Z}^3\ar@{->>}[ll] } \hspace{1cm}\mbox{ and }\hspace{1cm}
%		\xymatrix{  & \mathbb{Z}^8_+\ar@{->>}[ld]\ar@{->>}[rd]  & \\ C_{P} & & C_Q.\ar@{->>}[ll] }
%	\end{eqnarray*}
There are two cases how $\mathcal{A}_{(1,2,1),3}$ could be a toric subvariety of $\mathcal{U}_{3,3}$. Either it is a toric divisor or the inclusion is a toric morphism.
	A careful technical analysis shows that no change of coordinates in $\p^7$ can simultaneously make the ideals of both $\mathcal{A}_{(1,2,1),3}$ and $\mathcal{U}_{3,3}$ binomial.
\end{example}
Despite previous Example, the axis paths variety turns out to be toric in many cases.

\begin{proposition}\label{pro:partial results on toricness}
	Let $d=\max(\nu_1,\ldots,\nu_m)$. Then
\begin{enumerate}
%	\item The length of $\nu$ is large compared to $d$.
	\item $\AAA_{\nu, 3}$ is toric for every $d\le 2$. 
	\item $\AAA_{\nu, 2}$ is toric for every $d \leq 3$.
%		$\mathcal{A}_{(1,2,1), 3}$ and $\mathcal{A}_{(1,2,1,2),3}$ are toric.
	\item $\mathcal{A}_{(1,2,\dots, d,j), 2}$ is toric for every $j\in\{1,\dots, d-1\}$.
\end{enumerate}
\begin{proof}It is not difficult to check all possible cases for the first two items. The computation can be found in \cite{CGM19}.
%	\begin{enumerate}		\item The sequences $\nu=(1,2,1)$ and $\nu=(1,2,1,2)$ are addressed in \cite{CGM19}. The next case $\AAA_{(1,2,1,2,1),3}=\U_{2,3}$ already fills the universal variety, and we know that it is toric.
%		\item When we apply the change of coordinates \[\left( \begin{matrix}		\tilde{s}_{11} & \tilde{s}_{12} &\tilde{s}_{13}\\		\tilde{s}_{21} & \tilde{s}_{22} &\tilde{s}_{23}\\		\tilde{s}_{31} & \tilde{s}_{32} &\tilde{s}_{33}		\end{matrix}\right) =\left(  \begin{matrix}		2s_{11} & s_{12}+s_{21} &s_{13}+s_{31}\\s_{21} & 2s_{22} &s_{23}+s_{32}\\		s_{31} &s_{23} &2s_{33}		\end{matrix}\right) .		\]
%		on $(\R^3)^{\otimes 2}$, the ideal of $\mathcal{A}_{\nu,2}$ becomes a toric ideal for every $d\le 3$.
	For the third item, we have $m=d+1$. Let us  apply the change of coordinates		\[		\begin{cases}		\tilde{a}_j = a_j + a_{d+1}\\		\tilde{a}_i = a_i \text{ for } i\neq j		\end{cases}\]	on the domain $\R^{d+1}$. By Lemma \ref{lem:CombDescription}, we know that the entries of $\sigma^{(2)}(X)$ are
	\[
	\sigma^{(2)}(X)_{il}=\begin{cases}
	a_ia_l=\tilde{a}_i\tilde{a}_l & \mbox{for $i,l\neq j$}\\
	a_i(a_j+a_{d+1})=\tilde{a}_i\tilde{a}_j& \mbox{for $i<l=j$}\\
	(a_j+a_{d+1})^2=(\tilde{a}_j)^2& \mbox{for $i=l=j$}\\
	a_{i}a_{d+1}=\tilde{a}_i\tilde{a}_{d+1}& \mbox{for $i>l=j$}\\
	0&\mbox{for $i=j<l$}\\
	a_ja_i&\mbox{for $i=j>l$}.
	\end{cases}\]
Finally, if we perform the change of coordinates $\sigma^{(2)}(X)_{jd}\mapsto \sigma^{(2)}(X)_{jd}+\sigma^{(2)}(X)_{j,d+1}$, the map becomes monomial.
	\end{proof}
%\todo[inline]{We could extend item 3 to the case of only one repeated entry, possibly repeated more than twice and in any positions.}
\end{proposition}

There are further cases in which we can prove that the axis paths variety is toric. %Moreover, when $\AAA_{\nu,k}$ is toric we can add a new entry to $\nu$ and the new variety is still toric.
\begin{lemma}\label{lem:adding extra letter} 
%Let $\nu$ be a shape and let $l\in\{1,\dots,d\}$ be a letter not appearing in $\nu$. Let $\nu l$ denote the sequence $(\nu_1,\dots, \nu_m,l)$. 
If the ideal of $\AAA_{\nu,\leq k}$ is binomial after a linear change of coordinates, then the same holds for the ideal of $\AAA_{\nu  (d+1),k}$.
	\begin{proof}Let $X$ be an axis path of shape $\nu$ and sequence of length $(a_1,\dots,a_m)$. Since $d+1$ does not appear in $\nu$, in the set partition $\pi_{\nu (d+1)}$ the $(d+1)$-st block only contains the entry $a_{m+1}$. Therefore, by Lemma \ref{lem:CombDescription},
	\begin{eqnarray*}
	\sigma(X)_{i_1\dots i_{\tilde{k}} (d+1)^{k-\tilde{k}}} =
	\begin{cases}
		\frac{1}{(k-\tilde{k})!}(a_{m+1})^{k-\tilde{k}}\sigma(X)_{i_1\dots i_{\tilde{k}}} & \mbox{if $d+1$ does not appear in } (i_1,\dots, i_{\tilde{k}})\\
		0 & \mbox{otherwise.}
\end{cases}
	\end{eqnarray*}

Hence, 	$\AAA_{\nu  (d+1),k}$, up to diagonal change of coordinates, is equal to the projectivization of $\AAA_{\nu,\leq k}$.
%By hypothesis, we have a way to write $\AAA_{\nu,k}$ as the image of a monomial map, so $\sigma(X)_{i_1\dots i_{\tilde{k}}}$ is a monomial. Therefore $\sigma(X)_{i_1\dots i_{\tilde{k}} l^{k-\tilde{k}}}$ is a monomial and then $\AAA_{\nu l,k}$ is toric as well.
\end{proof}
\end{lemma}
We finish this section with the following natural question.
\begin{question}\label{quest:toricness}
Is every variety $\AAA_{\nu,k}$ toric?
\end{question}
%In this paper we present several cases for which the answer to this question is positive. Some of them are particular cases and others include larger families.

Our results provide a positive answer in all relatively small cases. However, since we found no general technique, we expect that the answer may be negative. The main obstacle to provide a nontoric example is that a first potential candidate is already too large to be dealt with, either with ad hoc geometric arguments or general computational methods \cite{katthan2017polynomial}.

%For instance, using Lemma \ref{lem:adding extra letter}, the first item in Proposition \ref{pro:partial results on toricness} may be extended to $d\le 3$ since there are a few extra cases to check. We do not find this particular case illustrative since we are interested in a more general answer. For that, we still need to get a better understanding of the axis-parallel path varieties. 

\subsection{Dimension of $\AAA_{\nu,k}$}\label{sub:axisDim}
As we stressed in the introduction, a fundamental problem when we deal with signatures is to determine whether it is possible to recover the path from the signature. The best case scenario is injectivity, when we can uniquely reconstruct the path. In general this is too much to hope, so we focus on a weaker property, and we want to know when $g_{\nu,k}$ is  generically finite. Therefore we are interested in the dimension of the fibers, or equivalently in $\dim \AAA_{\nu,k}$. Since $\AAA_{\nu,k}\subset\U_{d,k}$ is the image of $\R^{\ell(\nu)}$ under a polynomial map, we have
	\[\dim\mathcal{A}_{\nu,k}\le\min\{\ell(\nu),\dim(\U_{d,k})\}.\]
	In general, the inequality can be strict. For instance, the dimension of $\AAA_{(1,2,1,2,3),3}$ is 4, while $\ell(\nu)=5$ and $\dim(\U_{3,3})=7$. 
In order to better study the dimension of our axis-parallel signature variety, we introduce the following definition.
\begin{definition}
	For $d,k\in\N$, we define the set
	\[N_{d,k}:=\{\nu\mid \max(\nu)\le d\mbox{ and $G_{\nu,k}$ is not generically finite}\}.\]
	We say that $\AAA_{\nu,k}$ is \emph{defective} if $\nu\in N_{d,k}$ for $d\ge \max(\nu)$. Otherwise, we say that $\AAA_{\nu,k}$ has the expected dimension. 
\end{definition}
It may seem that we made a choice in using $G_{\nu,k}$ instead of $g_{\nu,k}$ in the previous definition. However, the $k$-th signature of a generic piecewise smooth path $X$ defines all the previous signatures up to finitely many choices (see \cite[Section 6]{AFS18}). %MM:check generic 
Therefore,
%\todo[inline]{Laura: The rest of the subsection has notation on functions $f_{\nu,k}$ and $F_{\nu,k}$. I think that in our notation right now, they should be $g$ and $G$. I have changed them, let me know if I was wrong and I will undo it.}
	\[N_{d,k}=\{\nu\mid \max(\nu)\le d\mbox{ and $g_{\nu,k}$ is not generically finite}\}.\]
For the same reason, $N_{d,k+1}\subset N_{d,k}$ for every $k$.
%We want to understand elements of $N_{d,k}$. 
Each $N_{d,k}$ is a language, i.e.~a set of words. It would be very interesting to completely characterise it - cf.~Conjecture \ref{con:nondef}. 
First, we prove that $N_{d,k}$ is absorbing with respect to concatenation.

\begin{lemma}
	\label{lem:concatenation remains defective} Let $\nu_1,\nu_2$ be two shapes, let $\nu_1\nu_2$ be their concatenation and let $d=\max(\nu_1\nu_2)$. If $\nu_1\in N_{d,k}$, then both $\nu_1\nu_2$ and $\nu_2\nu_1$ belong to $N_{d,k}$.
	\begin{proof}
		Take a general point of $\AAA_{\nu_1\nu_2,\le k}$. It is of the form $\sigma^{\le k}(X)$ for some general axis-parallel path $X$. Since $X$ has shape $\nu_1\nu_2$, we can write it as a concatenation $X_1X_2$, for general paths $X_1$ of shape $\nu_1$ and $X_2$ of shape $\nu_2$. By hypothesis, $\nu_1\in N_{d,k}$, and so there exist infinitely many paths $Y$ of shape $\nu_1$ such that $g_{\nu_1,a}(Y)=g_{\nu_1,a}(X_1)$ for every $a\le k$. By Proposition \ref{cor:Chen for axis paths},
\[g_{\nu_1\nu_2,k}(YX_2)=
\sum_{a+b= k}g_{\nu_1,a}(Y)\otimes g_{\nu_2,b}(X_2)=
\sum_{a+b= k}g_{\nu_1,a}(X_1)\
\otimes g_{\nu_2,b}(X_2)=
g_{\nu_1\nu_2,k}(X),\]
hence the fiber containing $X$ is not finite and therefore $\nu_1\nu_2\in N_{d,k}$. In the same way it is possible to prove that $\nu_2\nu_1\in N_{d,k}$.
	\end{proof}
%	\todo[inline]{Laura: I have a question. What happens if $\nu_1\nu_2$ is too long and ends up filling the universal variety? Are we excluding this case here? I got asked this question during a talk, and I know we need to exclude that situation, but it's not clear to me if we do with the definition that we have for being defective.}
\end{lemma}

On the other hand, if $\AAA_{\nu,k}$ is not defective, then we can add a new letter at any point in $\nu$ and the resulting variety is still of expected dimension.

\begin{lemma}\label{lem:adding a new letter does not affect defectiveness} Let $\nu\notin N_{d,k}$ and let $d=\max(\nu)$. If we write $\nu=\nu_1\nu_2$ and we take a new letter $l>d$, then $\nu_1\cdot l\cdot\nu_2\notin N_{l,k}$.
	\begin{proof}
Let us pick a general element in the affine cone over $\AAA_{\nu_1\cdot l\cdot\nu_2,k}$.
First we know that, up to finite number of choices, we can identify the parameter $a$ associated to $l$, as $\sigma_{(l^k)}=\frac{a^k}{k!}$. We may also identify the other parameters as the signatures indexed by numbers from one to $d$ are the same for $\nu$ and $\nu_1\cdot l\cdot\nu_2$.
%\todo[inline]{Laura: I don't know how to prove this with the definition of $N_{d,k}$ that we have.}
	\end{proof}
\end{lemma}

We noticed several times that the universal variety can be obtained as $\AAA_{\nu,k}$ for sufficiently long $\nu$. Now we want to prove that this can be done in an efficient way, namely by using a sequence with as many entries as $\dim\U_{d,k}$.
\begin{lemma}\label{lem:filling path}
	For every $d$ and $k$, there is a shape $\nu$ such that $\max(\nu)=d$, $\AAA_{\nu,k}=\U_{d,k}$ and $\ell(\nu)=\dim\U_{d,k}$.
\begin{proof}
	By Remark \ref{rmk:fill the universal}, there is a $\nu'$ satisfying the first two properties. We build $\nu$ as a subsequence of $\nu'$ by deleting those entries that do not increase the dimension of the variety in the following way. If $\nu'$ satisfies the last requirement too, we are done. Suppose then that $\nu'=(\nu_1,\dots,\nu_s)$ for $s>\dim \U_{d,k}$. By construction, there is an index $i\in\{1,\dots, s-1\}$ such that $\AAA_{(\nu_1,\dots,\nu_{i}),k}=\AAA_{(\nu_1,\dots,\nu_{i+1}),k}$. If $i=s-1$, take $\nu':=(\nu_1,\dots, \nu_{s-1})$ and proceed by induction on the length $\nu'$. Otherwise, let $\nu:=(\nu_1,\dots,\widehat{\nu_{i+1}},\dots,\nu_s)$, where the entry with the hat is omitted. Consider the sequences
	\begin{align}
	\alpha :=(\nu_1,\dots,\nu_{i}),\nonumber\quad
	\alpha':=(\nu_1,\dots,\nu_{i+1}),\nonumber\quad
	\beta  :=(\nu_{i+2},\dots,\nu_s),\nonumber
	\end{align}
	so that $\nu'=\alpha'\beta$ and $\nu=\alpha\beta$. By Proposition \ref{cor:Chen for axis paths}, the map
	\[\sigma^{\le k}\colon\R^{\ell(\nu')}\to\AAA_{\nu',\le k}\]
	factors as
	\[\R^{\ell(\nu')}\to\AAA_{\alpha',\le k}\times\AAA_{\beta,\le k}\xrightarrow{\text{Segre}} \AAA_{\nu',\le k}.\]
	We also have another map
	\[\R^{\ell(\nu)}\to\AAA_{\alpha,\le k}\times\AAA_{\beta,\le k}\xrightarrow{\text{Segre}} \AAA_{\nu,\le k}.\]
	Since $\AAA_{\alpha,\le k} = \AAA_{\alpha',\le k}$, we conclude that $\AAA_{\nu,\le k}=\AAA_{\nu',\le k}$.	
	We continue the process until there is no such index $i$. The resulting subsequence of $\nu'$ satisfies the requirements of the statement. 
\end{proof}
\end{lemma}

As an easy consequence, we notice that if $d=2$ there is only one possible shape $(1,2,1,2,\dots)$. Since it is unique, it satisfies the properties of Lemma \ref{lem:filling path} and therefore $\AAA_{\nu,k}$ has always the expected dimension.

Now, we look back to our first defective example. The reason why $\AAA_{(1,2,1,2,3),3}$ is defective is that there is a subsequence $(1,2,1,2)$ such that $\AAA_{(1,2,1,2),3}=\U_{2,3}$ but $\dim(\U_{2,3})=3$. This means that we are filling a 3-dimensional universal variety with a sequence of length four. 
%Let us finish this section with a conjecture that generalizes that phenomenon. 
We conjecture that this behavior is the only obstruction to having the expected dimension. 
\begin{conjecture}\label{con:nondef}
Let $\nu=(\nu_1,\dots,\nu_m)$ be a sequence with $d=\max(\nu)$. Then $\AAA_{\nu,k}$ is defective if and only if there is a subsequence $\nu^\prime=(\nu_i,\dots, \nu_{i+r})$ of $\nu$ with $d^\prime=\max(\nu^\prime)$, such that $\AAA_{\nu^\prime,k}=\U_{d^\prime,k}$ but $r+1>\dim(\U_{d^\prime,k})$.
\end{conjecture}

\subsection{Determinant of axis-parallel signatures}
Let us start with an observation about the entries of the $k$-th signature for a special family of axis-parallel paths.
\begin{lemma}\label{lem:lone entries are factors of slices}
	Let $k\in\N$ and let $X$ be an axis-parallel path with shape $\nu$ and sequence of length $(a_1,\dots,a_m)$. If an entry $\nu_i$ of $\nu$ appears only once in the $j$-th block of $\pi$, then $a_i$ divides all the entries of each $j$-th slice of $\sigma^{(k)}(X)$. If moreover $k=2$, then $a_i^2\mid\det(\sigma^{(2)}(X))$.
\begin{proof}
	If we look at one of the $j$-th slices, we are fixing one of the indices of $\sigma^{(k)}(X)$ to be $j$. By Lemma \ref{lem:CombDescription}, every monomial of the sum is a multiple of $a_i$. When $k=2$, we are dealing with a square matrix in which both the $j$-th row and the $j$-th column of $\sigma^{(2)}(X)$ are multiples of $a_i$. In order to conclude that $a_i^2\mid\det(\sigma^{(2)}(X))$, it is enough to check that $a_i^2$ divides the diagonal entry $\sigma^{(2)}(X)_{jj}$, and this follows from Lemma \ref{lem:CombDescription}.
\end{proof}
\end{lemma}
%\begin{corollary}\label{cor:lone entries are  factors of the determinant}Let $X$ be an axis-parallel path of shape $\nu$ and sequence of length $(a_1,\dots,a_m)$. If an entry $\nu_i$ of $\nu$ appears only once in the $j$-th block of $\pi$, then $a_i^2\mid \det(\sigma^{(2)}(X))$.\begin{proof}
%	By Lemma \ref{lem:lone entries are factors of slices}, $a_i$ is a factor of each entry of both the $j$-th row and the $j$-th column of $\sigma^{(2)}(X)$. In order to conclude, we only need to show that $a_i^2$ is a factor of $\sigma^{(2)}(X)_{jj}$, and this follows from Lemma \ref{lem:CombDescription}.\end{proof}\end{corollary}
This is just a special case of a much more general phenomenon, corroborated by many experiments with Sage \cite{sage} and the code included in \cite{CGM19}. The determinant of the second signature of an axis-parallel path is always the square of a polynomial in $a_1,\dots,a_m$. This subsection is devoted to explaining this result and its consequences. In order to correctly state it, we need some definitions. 
\begin{definition}\label{def:good shape}
For any shape $\nu=(\nu_1,\dots,\nu_m)$ with $d=\max \nu$, we say that a subsequence $\mu=(\nu_{i_1},\dots,\nu_{i_d})$ is a \emph{good subshape} if $i_1<\dots <i_d$ and $\{\nu_{i_1},\dots,\nu_{i_d}\}=\{1,\dots, d\}$. Moreover, we define the \emph{sign} of a good subshape, $\sgn\mu$, as the sign of the permutation $(\nu_{i_1},\dots,\nu_{i_d})\in S_d$.
\end{definition}
Without loss of generality we assumed that $\{1,\dots,d\}=\{\nu_1,\dots,\nu_m\}$, so good subshapes always exist.
\begin{definition}\label{def:Pdet}
Let $X$ be an axis path of shape $\nu$ and let $a_i$ be the parameter associated to $\nu_i$. We define the degree $d$ homogeneus polynomial $P(a)\in \CC[a_1,\dots,a_m]$ by 
$$P(a):=\sum_{\mu=(\nu_{i_1},\dots,\nu_{i_d})}(\sgn \mu) \prod_{j=1}^d a_{i_j},$$ 
where the sum is taken over all good subshapes $\mu$ of $\nu$. Moreover, we denote  by $\det_2 \nu\in \CC[a_1,\dots,a_m]$ the determinant $\det\left(\sigma^{(2)}(X)\right)$. By Lemma \ref{lem:CombDescription}, it is homogeneous of degree $2d$.
\end{definition}

\begin{theorem}\label{thm:det is a square}
With the notation from Definition \ref{def:Pdet}, we have $$2^d\det\left(\sigma^{(2)}(X)\right)=P(a)^2.$$
\end{theorem}
The proof we found is quite technical and we present it in Appendix \ref{app:proof}. However, Theorem \ref{thm:det is a square} has several interesting consequences. Since we can write $\U_{d,2}$ as an axis paths variety, we now know that the determinant of the signature matrix of any path is a square. In particular, if $X$ is any path, then $\det\sigma^{(2)}(X)\ge 0$. 
It is even more interesting to think in terms of the shuffle identity.
\begin{example}\label{example:determinant 2x2 is a shuffle square} Let $X\colon [0,1]\to\R^2$ be any path. By Lemma \ref{lem:shuffle identity},
\begin{align*}
\det(\sigma^{(2)}(X))&=\det\left( \begin{matrix}
\langle \sigma(X), 11\rangle &\langle \sigma(X), 12\rangle\\
\langle \sigma(X), 21\rangle &\langle \sigma(X), 22\rangle
\end{matrix}\right)\\
& = \langle \sigma(X), 11\rangle\cdot\langle \sigma(X), 22\rangle-\langle \sigma(X), 12\rangle \cdot\langle \sigma(X), 21\rangle\\
& = \langle \sigma(X), 11\shuffle 22-12\shuffle 21\rangle\\
& = \frac{1}{4}\langle \sigma(X), (12-21)^{\shuffle 2}\rangle.
\end{align*}
\end{example}
In this case it was not too difficult to realize that $4(11\shuffle 22-12\shuffle 21)=(12-21)^{\shuffle 2}$ and therefore that $\det(\sigma^{(2)}(X))$ is a square. %For $d=3$, an explicit computation shows that if $X:[0,1]\to\R^3$ is any path, then$\det\sigma^{(2)}(X)=\langle \sigma(X), (123-132+231-213+312-321)^{\shuffle 2}\rangle$. 
Note that $12-21$ is twice the L\'evy area defined by the path. This is a general behavior.
 \begin{definition}
	Let $S_d$ be the symmetric group in $d$ elements. Define
	\[
	\inv_d=\sum_{\rho\in S_d}\sgn(\rho)\rho(1)\dots\rho(d)\in T(\R^d).\]
	It is the sum of $d!$ words of length $d$, so it is a degree $d$ element of the shuffle algebra.
\end{definition}
On one hand, this has a geometric meaning: as illustrated in \cite[Section 3.2]{DR18}, $\inv_d$ has an interpretation in terms of the \emph{signed volume} of the convex hull of the path. On the other hand, there is a relation with signatures.
\begin{lemma}\label{lem: invariant and good underlinings}
	Let $a=(a_1,\dots,a_m)\in\R^m$ and let $X\colon [0,1]\to\R^d$ be the axis path corresponding to the sequence of lengths $a$ and shape $\nu$. Then $P(a)=\langle\sigma(X),\inv_d\rangle$.
	\begin{proof}
		By Lemma \ref{lem:CombDescription}, if  $i_1,\dots, i_d$ are pairwise disjoint then $\langle\sigma(X),i_1\dots i_d\rangle$ equals the sum of products $\prod_{l=1}^d a_{j_l}$ where $\nu_{j_l}=i_l$ and $j_1<\dots<j_d$. In particular, it is a subsum in the definition of $P(a)$ corresponding to the good subshapes associated to the permutation $(i_1,\dots,i_d)$. Summing up over all possible permutations $(i_1,\dots,i_d)$ with signs we obtain the statement of the lemma.
	\end{proof}
\end{lemma}
It follows that Theorem \ref{thm:det is a square} has an important consequence on the shuffle algebra.
\begin{corollary}\label{cor:shuffle determinant is a shuffle square}
%	Let $\mathcal{S}=(T(\R^d),\shuffle,e)$ be the shuffle algebra. Let
	Consider the matrix
	$$A=\left( \begin{matrix}
	11 & 12 &\dots &1d\\21 & 22 &\dots &2d\\
	\vdots & \vdots &\ddots &\vdots\\
	d1 & d2 &\dots &dd\\
	\end{matrix}\right) %\in\mathcal{S}\otimes\mathcal{S}.
	$$with coefficients in the shuffle algebra $(T(\R^d),\shuffle,e)$. Let $\det_\shuffle(A)$ be its determinant, where the product is the shuffle. 
Then $2^d\det_\shuffle(A)=(\inv_d)^{\shuffle 2}$.
\begin{proof}
	It is enough to prove that $\langle T,(\inv_d)^{\shuffle 2}\rangle=2^d \langle T,\det_{\shuffle}(A)\rangle$ for every $T\in T((\R^d))$. Since the linear span of $\G(\R^d)$ is the whole $T((\R^d))$, by linearity we just have to show that such equality holds when $T\in\G(\R^d)$. Let us pick then $T\in\G(\R^d)$. Then there is an axis path $X$, corresponding to the sequence of lengths $a\in\R^m$, such that $\langle T,(\inv_d)^{\shuffle 2}\rangle=\langle \sigma(X),(\inv_d)^{\shuffle 2}\rangle$ and  $\langle T,\det_{\shuffle}(A)\rangle=\langle\sigma(X),\det_{\shuffle}(A)\rangle$. By Theorem \ref{thm:det is a square} and Lemma \ref{lem: invariant and good underlinings},
	\[
	2^d\langle \sigma(X),\left.\det\right._{\shuffle}(A)\rangle=2^d\det(\sigma^{(2)}(X))=P(a)^2=\langle \sigma(X),(\inv_d)^{\shuffle 2}\rangle.\qedhere\]
\end{proof}
\end{corollary}
%Since the shuffle algebra is commutative, the element $v$ is unique up to a sign. Since it is the only degree $d$ invariant of $\mathcal{S}$ under the action of $\GL(\R^d)$, it is denoted by $\inv_d$ in \cite{DR18}.

Even though the techniques we use are based on the combinatorics of $\AAA_{\nu,k}$, Corollary \ref{cor:shuffle determinant is a shuffle square} has nothing to do with axis paths, nor with signatures at all. It would be very interesting to find a generalization for $k>2$. In our opinion, this can give new insight on the properties of the shuffle algebra.

\appendix\label{sec:appendix}
%\section{Toric geometry}\label{app:toric geometry}
%In this appendix we recall the basic methods of toric geometry used in the article. We do it for convenience of the reader referring for a 

%more comprehensive description of the results to \cite{CLS11, BerndBook, Fulton, Oda, topicsonTV}.

\section{Proof of Theorem \ref{thm:det is a square}}\label{app:proof}
For this proof we fix an axis path $X$ with shape $\nu\in\R^m$ and sequence of lengths $a\in\R^m$. The first step will be to reduce the problem to the case in which every entry of $\nu$ appears at most twice. Let us start with an observation.
\begin{remark}\label{rmk: we can relabel}
	The symmetric group $S_d$ acts on $\R^d$ by permuting the basis elements. In the same way, it acts on the sequence $\nu$, on the path $X$ and thus on the polynomials $P$ and $\det_2\nu$. The determinant is invariant% with respect to that action, as simultaneous permutations of rows and columns does not change the sign. On the other hand
	, while the sign of $P$ changes according to the sign of the permutation. However, $P^2$ is invariant. This means that we are allowed to relabel the entries of $\nu$. In other words, we can always pick $\sigma\in S_d$ and replace $(\nu_1,\dots,\nu_m)$ with $(\sigma(\nu_1),\dots,\sigma(\nu_m))$.
\end{remark}
In order to prove that two polynomials are equal, it is enough to prove that they have the same coefficient on each monomial. 
\begin{notation}
For a polynomial $Q$ and a monomial $M$ let $Q_{|M}$
be the coefficient of $Q$ corresponding to $M$. 
\end{notation}
Since both $P^2$ and $\det_2\nu$ are homogeneous of degree $2d$, we only have to take care of degree $2d$ monomials.

\begin{lemma}\label{lem: no triple entries}
	Suppose that there are three (not necessarily distinct) indices $i,j,l\in\{1,\dots,m\}$ such that $\nu_i=\nu_j=\nu_l$.
	\begin{enumerate}
		\item Let $M\in\C[a_1,\dots,a_m]_{2d}$ be such a monomial that $a_ia_ja_l\mid M$. Then $P^2_{|M}=(\det_2\nu)_{|M}=0$.
		\item Let $\nu'$ be the sequence obtained from $\nu$ by removing $\nu_l$ and let $P'$ be the associated polynomial. If $M\in\C[a_1,\dots,\hat{a_l},\dots,a_m]_{2d}$ is a monomial, then $P'_{|M}=P_{|M}$ and $\det_{2}\nu_{|M}=\det_{2}\nu'_{|M}$.
	\end{enumerate}
	\begin{proof}
		Thanks to remark \ref{rmk: we can relabel}, we may assume $\nu_i=\nu_j=\nu_l=1$. By Lemma \ref{lem:CombDescription}, the variables $a_i,a_j,a_l$ only appear in the first row and in the first column of the signature matrix. Furthermore, among the monomials of the diagonal entry $\sigma(X)_{11}$ there are degree at most 2 monomials in the variables $a_i,a_j,a_l$, while away from the diagonal all monomials contain only one of these variables, with exponent 1. Therefore $a_ia_ja_l$ does not divide any monomial appearing in $\det_2\nu$. On the other hand, every good subshape $\mu$ of $\nu$ contains the entry 1 exactly once, so each monomial of $P$ contains at most one among $a_i,a_j,a_l$, and with exponent 1. In $P^2$ there can be monomials containing the product of two among $a_i,a_j,a_l$, but not containing all three of them. The second statement follows from the definitions of $P$ and $\det_2\nu$.
	\end{proof}
\end{lemma}
Thanks to Lemma \ref{lem: no triple entries}, from now on we can restrict our attention to sequences $\nu$ with no triple entries. Now we want to take care of the case in which an entry appears only once. We find it useful to focus on a particular monomial.
\begin{definition}
	Assume that $\nu$ has no triple entries. If $\mu=(\nu_{i_1},\dots,\nu_{i_\ell})$ is a subsequence of $\nu$, we set
	\[e(i_j)=
	\begin{cases}
	1 & \mbox{if $\nu_{i_j}$ appears twice in } \mu\\
	2 & \mbox{if $\nu_{i_j}$ appears onece in } \mu\\
	\end{cases}\]
	and we define the monomial
	\[M_{\mu}=\prod_{j=1}^{\ell}a_{i_j}^{e(i_j)}\in\C[a_1,\dots,a_m].\]
\end{definition}

%In particular, we will see how it pays off to focus on the monomial $M_\nu$. 
Since we are assuming that $\nu$ has no triple entries, $\deg(M_\nu)=2d$.

\begin{lemma}\label{lem: it is enough to look at M} %The monomial $M_\nu$ is a summand of both $P^2$ and $\det_2\nu$. 
If a monomial $M$ appears in either $P^2$ or $\det_2\nu$ and it is divisible by all variables $a_1,\dots,a_m$, then $M=M_\nu$.
	\begin{proof}
		Let $M$ be such a monomial. By Lemma \ref{lem: no triple entries}, $M\mid M_\nu$. Indeed, if $\nu_i$ appears in $\nu$ at most once, then $a_i^3\nmid M$. If $\nu_i=\nu_j$ appears exactly twice then $a_ia_j\mid M$, but $a_i^2,a_j^2\nmid M$.
		As the degrees are the same, equality follows. 
	\end{proof}
\end{lemma}
 The next results will explain what happens to the coefficient of $M_\nu$ when $\nu$ has a non-repeated entry.
\begin{lemma}\label{lem:nonrepeated entry P}
	Assume $\nu$ contains at least one non-repeated entry. Suppose that the last one occurs in $\nu_{m-l}$. Let $\nu_0$ be the sequence obtained from $\nu$ by discarding $\nu_{m-l}$ and let $P_0$ be the associated polynomial. Then
	$$(P(a)^2)_{|M_\nu}=(-1)^l(P_0(a)^2)_{|M_{\nu_0}}.$$
	\begin{proof}
		By Remark \ref{rmk: we can relabel}, we can assume that $\nu_{m-l}=d$. Since $M_\nu$ is a monomial of $P^2$, to compute $P^2_{|M_\nu}$ we have to consider the contribution of all pairs $(\alpha,\beta)$ of good subshapes of $\nu$. %But $M_\nu$ is the monomial corresponding to the whole sequence. 
		Let $(\alpha,\beta)$ be a pair contributing to $M_\nu$. The variable of any non-repeated entry - including $d$ - has to appear in both $\alpha$ and $\beta$. On the other hand, the variable of an entry that appears twice appears once in $\alpha$ and once in $\beta$. Therefore there is a bijection between pairs $(\alpha,\beta)$ of good subshapes of $\nu$ contributing to $M_\nu$ and pairs $(\alpha_0,\beta_0)$ of good subshapes of $\nu_0$ contributing to $M_{\nu_0}$, where $\alpha_0$ is the sequence obtained from $\alpha$ by removing $d$, and $\beta_0$ is defined in a similar way from $\beta$. If we set now $a$ (resp. $b$) to be the number of entries of $\alpha$ (resp $\beta$) after the discarded one, then 
		\begin{align*}
	(P(a)^2)_{|M_\nu}&=\sum_{(\alpha,\beta)}\sgn(\alpha)\sgn(\beta)=\sum_{(\alpha_0,\beta_0)}(-1)^a\sgn(\alpha_0)(-1)^b\sgn(\beta_0)\\
	&=(-1)^l\sum_{(\alpha_0,\beta_0)}\sgn(\alpha_0)\sgn(\beta_0)=(-1)^l(P_0(a)^2)_{|M_{\nu_0}}.\qedhere
		\end{align*}
	\end{proof}
\end{lemma}
Before we prove the analogous statement for the polynomial $\det_2\nu$, we introduce a combinatorial interpretation of the coefficient ${\det}_2(\nu)_{|M_\nu}$. %Let us associated to $\nu$ a directed graph. Consider a set of $m$ vertices, ordered from left to right, corresponding to $(\nu_1,\dots,\nu_m)$. All the edges will only go from left to right. For every monomial appearing in the $(i,j)$-th entry of the signature matrix, If $i\neq j$, then such a monomial is the product of two variables corresponding to $i$ and $j$, where $i$ preceeds $j$ in $\nu$. This will be represented by an arrow in the word $\nu$ from $i$ to $j$.for $i=j$ either as above or half of the square of the variable corresponding to $i$. This will be represented as a loop over the corresponding variable.
By Laplace expansion, the contributions to ${\det}_2(\nu)_{|M_\nu}$ come from directed graphs (with possible loops) with vertices corresponding to symbols in $\nu$ and:
\begin{itemize}
	\item if a symbol $i$ appears twice in $\nu$ then precisely one such vertex is outgoing and one is incoming;
	\item if a symbol $i$ appears once in $\nu$ then it has degree two and has one incoming and one outgoing edge - this is the only case where a vertex can be a loop;
	\item all edges are from left to right.
\end{itemize}
Each such graph contributes $\pm \frac{1}{2^m}$ where $m$ is the number of loops and the sign is the sign of the corresponding permutation.
\begin{example}
	Let $\nu=(1,2,1)$. We have two possible graphs. One with edges $(1,2)$ and $(2,1)$. It encodes the transposition $(1,2)$ and contributes with $-1$. The other one has an edge $(1,1)$ and a loop over two. It encodes the identity permutation and contributes $\frac{1}{2}$. 
\end{example}
\begin{lemma}\label{lem:nonrepeated entry det}
	%Suppose there are $l$ symbols in $\nu$ appearing after $d$ (which appears just once).% Suppose the variable $a$ corresponds to the symbol $d$. 
	With the same hypothesis as in Lemma \ref{lem:nonrepeated entry P},
	$${\det}_2(\nu)_{|M_\nu}=\frac{(-1)^l}{2}\cdot {\det}_2(\nu_0)_{|M_{\nu_0}}.$$
\end{lemma}
\begin{proof}
	
	We look at the possible cases for the position of $d$. 
	
	If $d$ is the last symbol then the $d$-th column of the second signature matrix has only one nonzero entry, which equals $\frac{1}{2}$ times the square of the variable associated to $d$. Applying Laplace expansion we obtain the formula in this case.  
	
	We now assume that $d$ is the last but one entry. The contributing graphs are of two types. 
	
	If $d$ is a loop then (by forgetting the loop) we obtain graphs on $\nu_0$. Hence, the contribution of those graphs is $\frac{1}{2}({\det}_2(\nu_0))_{|M_{\nu_0}}$.
	
	If $d$ is not a loop, it must have an incoming edge, say $(a,d)$ and an outgoing edge $(d,b)$, where $b$ must be the last symbol in $\nu$. We may replace this by an edge $(a,b)$. In the cycle presentation of the permutation we removed $d$ from one cycle, i.e.~changed the sign of the permutation. As we did not change the number of loops the contribution equals $-({\det}_2(\nu_0))_{|M_{\nu_0}}$. Summing the two contributions we obtain the result in this case. 
	
	We may now assume that $d$ is the last letter in $\nu$ that appears once and further there are at least two symbols after it, i.e.~$\nu=\dots dab\dots$. Let $\nu'=\dots abd$ be $\nu$ with $d$ moved two places forward. We will prove that $({\det}_2(\nu))_{|M_{\nu}}=({\det}_2(\nu'))_{|M_{\nu'}}$. By induction this will finish the proof of the lemma.
	
	We may identify graphs contributing to $({\det}_2(\nu))_{|M_{\nu}}$ with those contributing to $({\det}_2(\nu'))_{|M_{\nu'}}$ with the exception of graphs for which there is an edge from $d$ either to $a$ or $b$. Further, if there is an edge $(d,a)$ and the considered $b$ is an outgoing vertex, the encoded permutation is $x\rightarrow d \rightarrow a$, $b\rightarrow c$ for some $x$ and $c$. We may associate to it a graph on $\nu'$ with edges $(x,a),(b,d),(d,c)$. This does not change the sign of permutation. We proceed in the same way if there is an edge $(d,b)$ and $a$ is outgoing. The remaining graphs are those where there is an edge from $d$ to $a$ or $b$ and the other vertex is incoming. We will prove that the sum of contributions of such graphs equals zero. 
	
	Consider a graph with edges $(d,a),(c,b)$ for some $c$. We associate to it a graph with edges $(c,a),(d,b)$. In the cycle decomposition of the permutation this operation either joins two cycles or decomposes one cycle into two, i.e.~changes the sign of the permutation. In particular, the contribution of each pair equals zero. 
	
	We also see that we obtain all graphs on $\nu'$, apart from those for which there is an edge from $a$ or $b$ to $d$ and the other vertex is outgoing. Just as above one can show that the contribution of such graphs equals zero.
	%
	%
	%Idea:
	% Otherwise move d two to the right: pair not pairable graphs.
\end{proof}
We can now reduce to the case in which every entry of $\nu$ appears exactly twice.
\begin{corollary}\label{cor: reduce to double entries} Assume Theorem \ref{thm:det is a square} holds for sequences in which every entry appears exactly twice. Then it holds for every sequence.%	If the sequence $\nu$ has a non-repeated entry, then	$$P(a)^2=2^d{\det}_2\nu.$$
\begin{proof}
	We compare the coefficients of monomials in both polynomials. If a monomial is different from $M_\nu$, then it does not contain one of the variables. In this case Lemma \ref{lem: it is enough to look at M} allows us to consider a subsequence of $\nu$ and conclude by induction on $\ell(\nu)$. Consider then $M_\nu$. By repeatedly applying Lemmas \ref{lem:nonrepeated entry P} and \ref{lem:nonrepeated entry det}, we can discard all non-repeated entries and conclude by hypothesis.
\end{proof}\end{corollary}
From now on we may assume that each of the letters $1,\dots,d$ appears exactly twice in $\nu$. In particular, $d=2m$. As in the proof of Corollary \ref{cor: reduce to double entries}, we can assume, by induction on $\ell(\nu)$, that $P^2_{|M}=\det_2\nu_{|M}$ for every monomial $M\neq M_\nu$. In order to prove that the two polynomials coincide, it is enough to find a point $q\in\R^{2m}$ such that $M_\nu(q)\neq 0$ and $P^2(q)=\det_2\nu(q)$. Since we suppose that every entry of $\nu$ appears exactly twice, we can define $q\in\R^{2m}$ by
\[
q_i=\begin{cases}
1 & \mbox{if $\nu_i$ appears for the first time in the $i$-th entry of } \nu,\\
-1 & \mbox{if $\nu_i$ appears for the second time in the $i$-th entry of } \nu.
\end{cases}\]
The following is a straightforward consequence of Lemma \ref{lem:CombDescription}.
\begin{lemma}
	The matrix $\sigma^{(2)}(q)$ is skew-symmetric. If $i\neq j$, then
	\[
	\sigma^{(2)}(q)_{ij}=\begin{cases}
-1  & \mbox{if the subsequence of $\nu$ with the symbols $i$ and $j$ is } jiji,\\
1 & \mbox{if the subsequence of $\nu$ with the symbols $i$ and $j$ is } ijij,\\
0 & \mbox{otherwise}.
	\end{cases}\]
\end{lemma}

To finish the proof it is enough to show that $P(q)^2=2^d\det_2(\nu)(q)$. As the second signature matrix is skew-symmetric, it is enough to prove that $P(q)=2^{d/2}\Pf_2(\nu)(q)$, where $\Pf$ is the Pfaffian.
We note that when $\nu =(1,1,2,2,\dots,d,d)$ the claim is easy as both sides equal zero. The following two lemmas allow to reduce any $\nu$ to this case, finishing the proof.

Let $\nu_i,\nu_{i+1}$ be two consecutive entries in $\nu$. Let $\nu'$ be the sequence obtained from $\nu$ by switching $\nu_i$ and $\nu_{i+1}$, and let $q'\in\R^{2m}$ be the corresponding point, defined in the same way we defined $q$. Let $\nu''$ be the sequence obtained from $\nu$ by removing both occurrences of the symbol $\nu_i$ and both occurrences of the symbol $\nu_{i+1}$. In a similar fashion we define $q''\in\R^{2m-4}$. %Define $q'\in $ to be the point analogous to $q$ (i.e.~we forget the four coordinates corresponding to $a$'s and $b$'s).
By Remark \ref{rmk: we can relabel}, we may assume without loss of generality that $\nu_i=d-1$ and $\nu_{i+1}=d$.

\begin{lemma}
	Let%\[
	%e(i)=\begin{cases}
	%1 & \mbox{if $\nu_i$ appears for the first time in the $i$-th entry of } \nu\\
	%-1 & \mbox{if $\nu_i$ appears for the second time in the $i$-th entry of } \nu
	%\end{cases}\]
	%and
	\[e(d)=\begin{cases}
		1 & \mbox{if $d$ appears for the first time in the $(i+1)$-st entry of } \nu\\
		-1 & \mbox{if $d$ appears for the second time in the $(i+1)$-st entry of } \nu.
	\end{cases}\]
	In a similar way we define $e(d-1)$. Then
	\begin{enumerate}
		\item $\Pf_2(\nu)(q)=\Pf_2(\nu')(q)+e(d-1)e(d)\Pf_2(\nu'')(q'')$ and
		\item $P_\nu(q)=P_{\nu'}(q)+2e(d-1)e(d)P_{\nu''}(q')$.
	\end{enumerate}
	
\end{lemma}
\begin{proof} 
\begin{enumerate}
	\item The change from $\nu$ to $\nu'$ changes the second signature matrix by replacing the lower right $2\times 2$ submatrix either from
	$\begin{pmatrix}
	0 & 0\\
	0&0
	\end{pmatrix}$
	to
	$\pm
	\begin{pmatrix}
	0 & 1\\
	-1 & 0
	\end{pmatrix}$
	or the other way round. The formula follows from the standard Laplace expansion for Pfaffians.
	\item The good underlinings that involve at most one of the exchanged $d-1$ and $d$ provide the same contribution both to $\nu$ and $\nu'$. It remains to investigate the contribution of good underlinings containing both $d-1$ and $d$.
	These contribute to $\nu$ and $\nu'$ with opposite signs and are exactly the contributions of $P_{\nu''}(q')$. The sign depends only on the property if the underlined variables we forget are taken with plus or minus.\qedhere
\end{enumerate}
\end{proof}

We conclude by induction on the number of permutations needed to transform $\nu$ to $1122\dots dd$ and the length of $\nu$.

 \bibliographystyle{alpha}
 \bibliography{biblio}
\end{document}